\documentclass{amsart}%
\usepackage{eurosym}
\usepackage{amssymb}
\usepackage{amsmath,amsthm}
\usepackage{lipsum}
\usepackage{graphicx}
\usepackage{amsfonts}
\usepackage{mathrsfs}
\usepackage{graphicx}
\usepackage{xcolor}
\usepackage[pagewise]{lineno}
\usepackage{amsmath}%
\newtheorem{theorem}{Theorem}[section]
\newtheorem{lemma}[theorem]{Lemma}
\newtheorem{proposition}[theorem]{Proposition}
\newtheorem{corollary}[theorem]{Corollary}
\theoremstyle{definition}
\newtheorem{definition}[theorem]{Definition}
\theoremstyle{remark}

\newtheorem{remark}[theorem]{Remark}

\numberwithin{equation}{section}

\begin{document}
\title[Number of invariant measures for random expanding maps]{On the number of invariant measures for random expanding maps in higher dimensions}
\author{Fawwaz Batayneh, Cecilia Gonz\'{a}lez-Tokman}
\address{School of Mathematics and Physics, The University of Queensland, St Lucia, QLD
4072, Australia}

\begin{abstract}
In \cite{J}, Jab\l o\'{n}ski proved that a piecewise expanding $C^{2}$
multidimensional Jab\l o\'{n}ski map admits an absolutely continuous invariant
probability measure (ACIP). In \cite{BL}, Boyarsky and Lou extended this
result to the case of i.i.d. compositions of the above maps, with an on
average expanding condition. We generalize these results to the (quenched)
setting of random Jab\l o\'{n}ski maps, where the randomness is governed by an
ergodic, invertible and measure preserving transformation. We prove that the
skew product associated to this random dynamical system admits a finite number
of ergodic ACIPs. Furthermore, we provide two different upper bounds on the
number of mutually singular ergodic ACIP's, motivated by the works of Buzzi
\cite{B} in one dimension and Gora, Boyarsky and Proppe \cite{GBP} in higher dimensions.

\end{abstract}
\maketitle

\bigskip
%\email{}
%\urladdr{}
%\date{}
%\subjclass{ }

\section{Introduction}

A fundamental problem in ergodic theory is to describe the asymptotic
statistical behavior of orbits defined by a dynamical system. In this
approach, one attempts to understand and quantify the different invariant
measures of the system, in particular those which have physical relevance.
This problem has been studied intensively for several classes of piecewise
smooth systems, starting with one dimensional deterministic systems in the key
paper \cite{LY} by Lasota and Yorke in 1973. In 2000, Buzzi \cite{B}
identified bounds on the number of physical measures for random compositions
of Lasota--Yorke maps. In higher-dimensional frameworks, including random
versions of \cite{GBMULTI,S,C,Thomine}, understanding and, specifically,
bounding the number of physical measures is still an unsolved problem. This
challenge is related to open questions in multiplicative ergodic theory,
regarding multiplicity of Lyapunov exponents. The focus of this work is on
investigating and bounding the number of physical measures for a class of
higher dimensional expanding-on-average random dynamical systems, where the
randomness is driven by a rather general type of ergodic process, including
but not limited to the i.i.d. case.

In this paper we study a class of discrete time dynamical systems in which, at
each iteration of the process, one of a collection of maps is selected and
applied. Ulam and von Neumann \cite{UN}, Morita \cite{M}, Pelikan \cite{P} and
Buzzi \cite{B} were among those who started working on such systems, which
have been named time dependent, random or non autonomous dynamical systems. In
general, there is no measure which is invariant under all these maps
simultaneously. Therefore, we instead consider random invariant measures which
are absolutely continuous with respect to Lebesgue (ACIPs), and their
associated marginals, which give rise to physical measures.

This work focuses on dynamical systems modeled by random compositions of
so-called Jab\l o\'{n}ski maps. These maps have been studied by several
researchers after the first paper \cite{J} by Jab\l o\'{n}ski in 1983. In
\cite{GBapp}, G\'{o}ra and Boyarsky used Jab\l o\'{n}ski transformations as a
model for interacting cellular systems. In \cite{BLappr}, Boyarsky and Lou
presented a method for approximating the ACIPs in \cite{J} by means of
approximating the transfer operator by finite dimensional operators, which is
a version of Ulam's conjecture in a multidimensional setting. In \cite{BLG},
Boyarsky, G\'{o}ra and Lou considered a larger class of $C^{2}$
transformations defined on a rectangular partition of the $n$ dimensional
cube. The authors approximated any such map by a sequence of Jab\l o\'{n}ski
transformations and proved that the sequence of invariant densities associated
with these Jab\l o\'{n}ski maps converges weakly in $L^{1}$ to the invariant
density associated with that map. In \cite{Bose}, Bose replaced the weak
approximation of the invariant density in \cite{BLappr} by strong
approximation using a compactness argument. The special case of random i.i.d.
Jab\l o\'{n}ski maps was studied in \cite{BL,Mathesis}.

Random Jab\l o\'{n}ski maps, introduced in Definition~\ref{def:JabMap}, are
defined by a collection of piecewise smooth maps $(f_{\omega})_{\omega
\in\Omega}$ defined on the state or phase space $I^{n}$, where $I=[0,1]$ and
$n\in\mathbb{N}$ is the dimension, equipped with the Borel sigma algebra of
measurable sets and the $n$ dimensional Lebesgue measure $m$. The family of
maps is assumed to satisfy an expanding-on-average condition.

Our approach relies on so-called transfer operators, acting on the space of
higher dimensional functions of bounded variation. Given a nonsingular map
$f$, its transfer operator $\mathcal{L}_{f}$ encodes information about the
application of $f$ and describes how densities, i.e. nonnegative integrable
functions with integral one, evolve in time. If a collection of points in
phase space is distributed according to a probability density function $h$,
and pushed forward by $f$, then the resulting collection of points will be
distributed according to a new density denoted by $\mathcal{L}_{f}(h)$ or
$\mathcal{L}_{f}h$.

The first appearance of one dimensional functions of bounded variation is due
to C. Jordan in 1881 in connection with Dirichlet's test for the convergence
of Fourier series. In 1905, G. Vitali gave the first definition of bounded
variation function in two dimensions. Later on, L. Tonelli observed that
Vitali's generalization was not the right generalization of the one
dimensional variation because it contains second order elements related to the
curvature of the graph rather than its area. In 1936, in a closer analogy to
the one dimensional variation, Tonelli introduced his generalization which
measures the length of the projection of the graph onto the vertical axis
counting multiplicities at least for continuous functions. Tonelli's
definition is more convenient for continuous functions since the definition
depends on the choice of the coordinate axes if the function is not
continuous. To solve this issue, in the same year, L. Cesari modified
Tonelli's definition by requiring the integrals in Tonelli's definition to be
finite for functions equal almost everywhere. This definition does not depend
on the coordinates even for discontinuous functions. Functions of bounded
variation in this sense were called bounded variation functions in the sense
of Tonelli--Cesari. However, the point of view which is popular these days and
adapted in most of the literature \cite{GBMULTI} as the most suitable
generalization of the one dimensional theory is due to De Giorgi and Fichera.
Krickeberg and Fleming independently showed that a bounded variation function
in the sense of Tonelli--Cesari has a vector measure as its distributional
gradient, thus obtaining the equivalence with De Giorgi's definition. For more
information on historical details about higher dimensional functions of
bounded variation, we refer the reader to \cite{AFP}.

In the deterministic case, an early use of transfer operators in the one
dimensional bounded variation setting is due to Lasota and Yorke, who in
\cite{LY} proved the existence of ACIPs for piecewise $C^{2}$ transformations
$f$ on $I$, with the assumption of a uniform expanding condition
$\inf|f^{{\prime}}|>1$. The authors exploited the fact that the transfer
operator corresponding to the point transformation under consideration has the
property of keeping the variation of the functions $h,\mathcal{L}_{f}%
h,\dots,\mathcal{L}_{f}^{n}h,\dots$ under control. This result was later on
referred to as Lasota-Yorke inequality. In \cite{J}, Jab\l o\'{n}ski
generalized the one dimensional work of Lasota and Yorke \cite{LY} to
piecewise continuous maps on the multidimensional cube $I^{n}$ with similar
type of uniform expanding condition on the rectangles of a rectangular
partition. The proof of this result was similar to the proof of Theorem $1$ in
\cite{LY}, but it uses the notion of variation of functions of several
variables due to Tonelli--Cesari, which we also use in this paper.

In higher dimensions, the situation is more challenging than in one dimension.
For example, in the general case, crucial difficulties come from the much
richer geometry which can arise from the phase space partitions, and from the
growing complexity of the partitions arising from the iterated dynamics. To
overcome these issues, one may impose conditions on the geometry of the
partitions and, roughly speaking, to ensure the amount of expansion is enough
to overcome the dynamical complexity. See, for example, the conditions given
in Theorem $4$ in \cite{GBP}, Theorem $3.1$ in \cite{C} and equation $(1.8)$
in \cite{Thomine}.

%This inequality was applied to prove the existence of ACIPs.
A number of authors have studied the existence of ACIPs for piecewise
expanding maps in higher dimensions. In \cite{GBMULTI}, G\'{o}ra and Boyarsky
proved the existence of ACIPs with densities of bounded variation for
piecewise $C^{2}$ transformations in $%
%TCIMACRO{\U{211d} }%
%BeginExpansion
\mathbb{R}
%EndExpansion
^{n}$ for domains with piecewise $C^{2}$ boundaries with the assumption that
where the $C^{2}$ segments of the boundaries meet, the angle subtended by the
tangents to these segments at the point of contact is bounded away from zero.
The case when the boundaries for which the angle mentioned may become zero
(i.e. the boundaries of partitions may contain 'cusps') is studied in
\cite{Ke,Adl} by Keller and Adl-Zarabi. In \cite{S}, Saussol developed a
Lasota-Yorke inequality for a class of piecewise expanding maps defined on a
compact subset of $%
%TCIMACRO{\U{211d} }%
%BeginExpansion
\mathbb{R}
%EndExpansion
^{n}$ and used it to prove the existence of a finite number of ACIPs with
densities in the Quasi-H\"{o}lder space. The author also provided an upper
bound on the number of these ACIPs. In \cite{C}, Cowieson extended the work of
G\'{o}ra and Boyarsky by establishing a simpler condition which guarantees the
existence of an ACIPs. The condition is that, the expansion must be greater
than the cut index defined in \cite[Section $2.2$]{C}. The author made some
statements about random perturbations of such maps in \cite[Theorem $3.2$]{C}.
In \cite{Thomine}, Thomine gave a sufficient condition shown in \cite[Equation
$(1.7)$]{Thomine} under which a piecewise $C^{1+\alpha}$ uniformly expanding
map admits a finite number of ACIPs. The author also compares his results with
the work of Saussol \cite{S} and Cowieson \cite{C}. Although no upper bounds
on the number of these ACIPs are explicitly given in \cite{Thomine}, the
author mentions that the results of \cite{S} could perhaps be adapted to his setting.

%This inequality was applied to prove the existence of ACIPs.
In the random one-dimensional case, in \cite{B}, Buzzi considers random
expanding-on-average Lasota--Yorke maps that have neither too many branches
nor too large distortion, and proves that the associated skew product
transformation possesses a finite number of mutually singular ergodic ACIPs,
each giving a family of random invariant measures with densities of bounded
variation. In \cite{AS}, Araujo and Solano proved existence of ACIPs for
random one dimensional dynamical systems with asymptotic expansion. Their work
can be seen as a generalization of the work of Keller \cite{KEONE} which
proves that for maps on the interval with finite number of critical points and
non-positive Schwarzian derivative, existence of absolutely continuous
invariant probability is earned by positive Lyapunov exponents. They also
prove similar results for higher dimensional random systems under the
assumption of slow recurrence to the set of discontinuities and/or
criticalities, which are of a certain non-degenerate type, shown in
\cite[Equation $(1.5)$]{AS}.

In\ \cite{BL}, Boyarsky and Lou studied the case of i.i.d. compositions of
Jab\l o\'{n}ski maps as defined in \cite{J}. The authors considered the
setting of a finite number of piecewise $C^{2}$ and monotonic Jab\l o\'{n}ski
maps $f_{1}$,\dots,$f_{l}$ where $l$ is a finite positive integer and%
\[
f_{k,i}(x_{1},\dots,x_{n})=\varphi_{k,i,j}(x_{i})\text{,}%
\]
for $(x_{1},\dots,x_{n})\in D_{k,j}$ where $D_{k,j}$ is the $j^{\text{th}}$
rectangle in the partition of $f_{k}$. The random map $f$ is defined by
choosing $f_{i}$ with probability $p_{i}$, where the $p_{i}$'s are positive
and add up to one. The authors assumed an expanding-on-average condition that
is, there exists a positive constant $0<\gamma<1$ such that%
\begin{equation}
\sum_{i=1}^{l}\sup_{j}\frac{p_{i}}{|\varphi_{k,i,j}^{\prime}(x_{i})|}%
\leq\gamma\text{,} \label{foexc}%
\end{equation}
for all $i=1,\dots,l$ and $(x_{1},\dots,x_{n})\in cl(D_{k,j})$ (the closure of
$D_{k,j}$) and proved that $f$ admits an ACIP with respect to the Lebesgue
measure. This measure has density $h$ which is a fixed point of the averaged
transfer operator $\mathcal{L}_{f}$ of $f$, given by%
\begin{equation}
\mathcal{L}_{f}=\sum_{i=1}^{l}p_{i}\mathcal{L}_{fi}\text{,} \label{fato}%
\end{equation}
where $\mathcal{L}_{fi}$ is the transfer operator of the corresponding map
$f_{i}$.

Our conditions on the maps are much more general than the ones in \cite{BL,P}.
In both articles, the maps driving the dynamics or defining the random orbits
are given by an i.i.d. process. Moreover, the maps must be chosen from a
finite set. However, in our situation the way of selecting the maps comes from
the base map $\sigma$ defined on a probability measure space $(\Omega
,\mathcal{F},\mathbb{P)}$. A difficulty of this setting is that there is no
known formula for an averaged transfer operator that corresponds to the one
described in \eqref{fato} in \cite{P,BL}. The way we overcome this obstacle,
as it has been done in \cite{B,FLQ,CQ}, is by developing a random Lasota-Yorke
inequality, Equation~\eqref{lyi}, which we use to prove several results in
this paper.

Quasi-compactness is one of the concepts which has played a key role in the
modern approach to investigate transport properties of random dynamical
systems through transfer operators \cite{FLQ,CQ}. In the case of autonomous
systems, this property was introduced in the work of Ionescu Tulcea and
Marinescu \cite{IM}. A bounded linear operator defined on a Banach space is
called quasi-compact if its spectral radius is strictly larger than its
essential spectral radius. The formulation of a non-autonomous analogue of the
quasi-compactness property goes back to Thieullen \cite{Thieullen}. It is now
widely known that quasi-compactness can be usually derived from Lasota-Yorke
type inequalities, and this is the route we pursue.

The quasi-compactness theorem of Ionescu Tulcea and Marinescu \cite{IM} is
used to provide spectral decompositions and properties in the case of
deterministic dynamical systems similar to the ones given in Section $3$ in
\cite{GBMULTI}. However, in the random case one instead uses Oseledets type
multiplicative ergodic theorems. In 1968, Oseledets multiplicative ergodic
theorem was first introduced by Oseledets \cite{Oseledets} in the context of
random multiplication of matrices. In its basic form, it describes the
asymptotic behavior of a product of matrices sampled from a dynamical system.
After that, different proofs were provided and different generalizations have
been developed and applied to transfer operator cocycles, see \cite{FLQ,CQ}.
In this paper we adapt Theorem $17$ from \cite{FLQ} to provide an Oseledets
splitting for random Jab\l o\'{n}ski maps.

When applicable, multiplicative ergodic theorems provide existence and
finiteness of random ACIPs. However, explicit bounds do not come directly from
this machinery. Despite some progress by Buzzi \cite{B} and Araujo--Solano
\cite{AS}, the question of how to find bounds on the number of ACIPs in random
dynamical systems is largely open. In \cite[Theorem $2$]{GBP}, G\'{o}ra,
Boyarsky and Proppe proved that, in their setting, the support of absolutely
continuous invariant measures is open Lebesgue almost everywhere. They used
this key fact to obtain their result \cite[Theorem $3$]{GBP} that the number
of ergodic ACIP's for deterministic dynamical systems modeled by
Jab\l o\'{n}ski transformations is at most equal to the number of crossing
points. We combine elements of their arguments with ideas from the one
dimensional work of Buzzi on random Lasota-Yorke maps (see Section $3$ in
\cite{B}) to develop a bound on the number of mutually singular ergodic ACIP's
for a class of admissible random Jab\l o\'{n}ski maps. Another bound is also
developed, and these bounds are compared in Section \ref{S:ex}.

This paper is structured as follows: in Section \ref{S:2}, we state the
definition of admissible random Jab\l o\'{n}ski maps, which involves the
formulation of an expanding-on-average condition motivated from the expanding
condition given in \eqref{foexc}. In Section \ref{S:3}, in Theorem \ref{quco},
we prove that this random map is quasi-compact and the maximal Lyapunov
exponent is indeed zero. In Section \ref{S:4}, in Theorem \ref{muer}, we prove
that the random invariant densities of admissible random Jab\l o\'{n}ski maps
are of bounded variation and equivariant. In Corollary \ref{coro1}, we prove
these densities belong to the leading Oseledets subspace and the number of
ergodic ACIPs with respect to the associated skew product is finite. Theorem
\ref{ph} is a probabilistic conclusion that shows that the marginals of the
measures in Theorem \ref{muer} are physical, which means that for Lebesgue
almost initial condition, the asymptotic long term behaviour of the
corresponding random orbit will be described by one of these physical
measures. In Section~\ref{S:bounds}, we establish upper bounds on the number
of mutually singular ergodic ACIPs for a class of admissible random
Jab\l o\'{n}ski maps, and present an example in Section \ref{S:ex}.

\section{Terminology and background}

\label{S:2}

\subsection{Preliminaries}

In this subsection we state the basic definitions and tools that will be used
throughout the paper.

\begin{definition}
Let $(\Omega,\mathcal{F},\mathbb{P)}$ be a probability space. A measurable
transformation $\sigma:\Omega\circlearrowleft$ is said to be
\emph{nonsingular} if
\[
\mathbb{P}(\sigma^{-1}(A))=0\text{,}%
\]
for all $A\in\mathcal{F}$ with $\mathbb{P}(A)=0$.
\end{definition}

\begin{definition}
Let $(\Omega,\mathcal{F},\mathbb{P)}$ be a probability space. A transformation
$\sigma:\Omega\circlearrowleft$ is said to be a \emph{measure-preserving
transformation} or, equivalently, $\mathbb{P}$ is said to be a $\sigma
-$\emph{invariant measure}, if
\[
\mathbb{P}(\sigma^{-1}(A))=\mathbb{P}(A)\text{,}%
\]
for all $A\in\mathcal{F}$.
\end{definition}

\begin{definition}
Let $(\Omega,\mathcal{F},\mathbb{P)}$ be a probability space. A nonsingular
transformation $\sigma:\Omega\circlearrowleft$ is said to be \emph{ergodic} if
for all $A\in\mathcal{F},$ with $\sigma^{-1}(A)=A,$ we have $\mathbb{P}(A)=0$
or $\mathbb{P}(\Omega\backslash A)=0.$
\end{definition}

\begin{definition}
Let $(X,\mathcal{B},%
%TCIMACRO{\U{b5}}%
%BeginExpansion
\mu
%EndExpansion
)$ be a measure space and $f:X\circlearrowleft$ a nonsingular transformation.
The unique operator $\mathcal{L}_{f}:L^{1}(X)\circlearrowleft$ satisfying the
dual relation%
\[
\int_{A}\mathcal{L}_{f}h(x)\mu(dx)=\int_{f^{-1}A}h(x)\mu(dx)\text{,}%
\]
for every $A\in\mathcal{B}$ and $h\in L^{1}(X)$ is called \emph{the transfer
or Perron-Frobenius operator} corresponding to $f$.
\end{definition}

%For a nonsingular transformation $f:I^{n}\circlearrowleft$, the transfer
%operator $\mathcal{L}_{f}:L^{1}(I^{n})\circlearrowleft$ is defined by the
%formula%
%\[
%\int_{A}\mathcal{L}_{f}h\mu(dx)=\int_{f^{-1}A}h\mu(dx)\text{,}%
%\]
%where $A\in\mathcal{B}$.
For $x=(x_{1},\dots,x_{n})$, if $A=\prod_{i=1}^{n}[0,x_{i}]$ in the above
definition, then differentiating both sides, we obtain%
\[
\mathcal{L}_{f}h(x)=\frac{\partial^{n}}{\partial x_{1}\dots\partial_{x_{n}}%
}\int_{f^{-1}(\prod_{i=1}^{n}[0,x_{i}])}h(y)\mu(dy)\text{,}%
\]
where $\partial x_{i}$ is the derivative with respect to $x_{i}$,
$i=1,\dots,n$. This formula can be seen in Section $2$ in \cite{BLG}. It is
well known that the transfer operator is linear, positive, contractive and
$\mathcal{L}_{f}h=h$\ if and only if the measure $\nu$ where $d\nu=hd\mu$ is
invariant under $f$, see \cite{A}.

Given sets $A_{i},$ $i=1,\dots,n$, denote the Cartesian product of the sets
$A_{i}$ by $%
%TCIMACRO{\dprod \limits_{i=1}^{n}}%
%BeginExpansion
{\displaystyle\prod\limits_{i=1}^{n}}
%EndExpansion
A_{i}=\{(a_{1},\dots,a_{n}):a_{i}\in A_{i},i=1,\dots,n\}$. For $i=1,\dots,n,$
let $P_{i}$ be the projection of $%
%TCIMACRO{\U{211d} }%
%BeginExpansion
\mathbb{R}
%EndExpansion
^{n}$ onto $%
%TCIMACRO{\U{211d} }%
%BeginExpansion
\mathbb{R}
%EndExpansion
^{n-1}$ given by
\begin{equation}
P_{i}(x_{1},\dots,x_{n})=(x_{1},\dots,x_{i-1},x_{i+1},\dots,x_{n})\text{.}
\label{projj}%
\end{equation}

We next describe the definition of the total variation of an integrable
function of several variables, due to Tonelli and Cesari, which was first used
in the context of transfer operators in \cite{BL,J}.

\begin{definition}
Consider the $n$ dimensional rectangle $A=$ $\ \
%TCIMACRO{\dprod \limits_{i=1}^{n}}%
%BeginExpansion
{\displaystyle\prod\limits_{i=1}^{n}}
%EndExpansion
[a_{i},b_{i}]$ where $a_{i},b_{i}\in%
%TCIMACRO{\U{211d} }%
%BeginExpansion
\mathbb{R}
%EndExpansion
$ and $a_{i}<b_{i}$ and a function $g:A\rightarrow%
%TCIMACRO{\U{211d} }%
%BeginExpansion
\mathbb{R}
%EndExpansion
$. For $i=1,\dots,n$, consider a real valued function $\underset{i}%
{\overset{A}{\mathbf{V}}}g$ of $(n-1)$ variables $(x_{1},\dots,x_{i-1}%
,x_{i+1},\dots,x_{n})$ given by%
\[
\underset{{\tiny
\begin{array}
[c]{c}%
a_{i}=x_{i}^{0}<x_{i}^{1}<\dots<x_{i}^{r}=b_{i}\\
r\in%
%TCIMACRO{\U{2115}}%
%BeginExpansion
\mathbb{N}%
%EndExpansion
\end{array}
}}{\sup}\sum_{k=1}^{r}\Big(|g(x_{1},\dots,x_{i}^{k},\dots,x_{n})-g(x_{1}%
,\dots,x_{i}^{k-1},\dots,x_{n})|\Big)\text{.}%
\]

For a measurable function $f\in L^{1}(A)$ and $i=1,\dots,n$, define
\[
\underset{i}{\overset{A}{\mathbb{V}}}f=\underset{{\tiny
\begin{array}
[c]{c}%
g=f\text{ a.e.}\\
\underset{i}{\overset{A}{\mathbf{V}}}g\text{ is measurable}%
\end{array}
}}{\inf}\int_{P_{i}(A)}\underset{i}{\overset{A}{\mathbf{V}}}gdm\text{,}%
\]

\end{definition}

\begin{definition}
where $P_{i}$ is defined in \eqref{projj}and let \emph{the total variation of
}$f$ be
\[
\overset{A}{\mathbb{V}}f=\underset{i=1,\dots,n}{\max}\underset{i}{\overset
{A}{\mathbb{V}}}f\text{.}%
\]
If the total variation$\overset{A}{\text{ }\mathbb{V}}f$ of $f$ on $A$ is a
finite, then $f$ is said to be of \emph{bounded variation on }$A$, and\ the
set of all such maps is denoted by $BV(A)$. For $f\in BV(A)$, the norm of $f$
is defined by $\Vert f\Vert_{BV}=\Vert f\Vert_{1}+\overset{A}{\mathbb{V}}f$.
\end{definition}

The space $BV(A)$ is a Banach space by Remark $1.12$ in \cite{G} and compactly
embedded in $L^{1}(A)$ by Corollary $3.49$ in \cite{AFP}.

\subsection{Random dynamical systems}

%%%

\begin{definition}
\emph{A random dynamical system} is a tuple $\mathcal{R}=(\Omega
,\mathcal{F},\mathbb{P},\sigma,\mathcal{X},\mathcal{L}\mathbb{)}$, where the
base $\sigma$ is an invertible measure-preserving transformation of the
probability space $(\Omega,\mathcal{F},\mathbb{P)}$, $(\mathcal{X},\Vert
\cdot\Vert)$ is a Banach space and $\mathcal{L}:\Omega\mathcal{\rightarrow
}L(\mathcal{X},\mathcal{X})$ is a family of bounded linear maps of
$\mathcal{X}$, called \emph{the generator}.
\end{definition}

For convenience, we let $\mathcal{L}_{\omega}:=\mathcal{L}(\omega)$. A random
dynamical system defines a cocycle, given by
\[
(k,\omega)\mapsto\mathcal{L}_{\omega}^{(k)}:=\mathcal{L}_{\sigma^{k-1}\omega
}\circ\dots\circ\mathcal{L}_{\sigma\omega}\circ\mathcal{L}_{\omega}\text{.}%
\]

Different regularity conditions may be imposed on the generator $\mathcal{L}$.
The following concept of $\mathbb{P}$-continuity, which was first introduced
by Thieullen in \cite{Thieullen}, will be used in the sequel.

\begin{definition}
Let $\Omega$ be a topological space, equipped with a Borel probability measure
$\mathbb{P}$ and let $Y$ be a topological space. A mapping $L:\Omega
\rightarrow Y$ is said to be $\mathbb{P}$\emph{-continuous} if $\Omega$ can be
expressed as a countable union of Borel sets such that the restriction of $L$
to each of them is continuous.
\end{definition}

In the rest of this work, we consider random dynamical systems whose
generators $\mathcal{L}:\Omega\mathcal{\rightarrow}L(\mathcal{X},\mathcal{X}%
)$, given by $\omega\mapsto\mathcal{L}_{\omega}$, are $\mathbb{P}$-continuous
and that $\Omega$ is a Polish space. That is, a complete separable metric space.

\begin{definition}
\emph{The index of compactness} (or \emph{Kuratowski measure of
noncompactness}) of a bounded linear map $A:\mathcal{X}\circlearrowleft$ is
\[
\Vert A\Vert_{ic(\mathcal{X})}=\inf\{r>0:A(B_{\mathcal{X}})\text{ can be
covered by finitely many balls of radius }r\}\text{,}%
\]
where $B_{\mathcal{X}}$ denotes the unit ball in $\mathcal{X}$.
\end{definition}

\begin{definition}
\label{le} Let $\mathcal{R}$ $=$ $(\Omega,\mathcal{F},\mathbb{P}%
,\sigma,\mathcal{X},\mathcal{L})$ be a random dynamical system. Assume that
$\int_{\Omega}\log^{+}\Vert\mathcal{L}_{\omega}\Vert d\mathbb{P(\omega
)<\infty}.$ For each $\omega\in\Omega,$ \emph{the maximal Lyapunov exponent
}$\lambda(\omega)$\emph{ for }$\omega$ is defined as%
\[
\lambda(\omega)=\lim_{k\rightarrow\infty}\frac{1}{k}\log\Vert\mathcal{L}%
_{\omega}^{(k)}\Vert\text{,}%
\]
whenever the limit exists. \emph{The index of compactness }$K(\omega)$\emph{
for }$\omega$ is defined as%
\[
\mathcal{K}(\omega)=\lim_{k\rightarrow\infty}\frac{1}{k}\log\Vert
\mathcal{L}_{\omega}^{(k)}\Vert_{ic(\mathcal{X})}\text{,}%
\]
whenever the limit exists.
\end{definition}

The following is established in \cite{CQ}.

\begin{remark}
\label{re4} If a random dynamical system $\mathcal{R}$ has an ergodic base
$\sigma$, then $\lambda$ and $\mathcal{K}$ in the previous definition are
$\mathbb{P}-$almost everywhere constant. We call these constants
$\lambda^{\ast}(\mathcal{R})$ and $\mathcal{K}^{\ast}(\mathcal{R})$, or simply
$\lambda^{\ast}$ and $\mathcal{K}^{\ast}$, if $\mathcal{R}$ is clear from the
context. It follows from the definition that $\mathcal{K}^{\ast}\leq
\lambda^{\ast}$. The assumption $\int_{\Omega}\log^{+}\Vert\mathcal{L}%
_{\omega}\Vert d\mathbb{P(\omega)<}\infty$ implies that $\lambda^{\ast}%
<\infty$.
\end{remark}

\begin{definition}
A random dynamical system $\mathcal{R}$ with an ergodic base $\sigma$ is
called \emph{quasi-compact} if $\mathcal{K}^{\ast}<\lambda^{\ast}$.
\end{definition}

The next proposition relates the maximal Lyapunov exponent $\lambda^{\ast
}(\mathcal{R})$ and the index of compactness $\mathcal{K}^{\ast}(\mathcal{R})$
of a random dynamical system $\mathcal{R}$ with the corresponding quantities
for $\mathcal{R}^{(n)}$ $=$ $(\Omega,\mathcal{F},\mathbb{P},\sigma
^{n},\mathcal{X},\mathcal{L}^{(n)})$, $n\in%
%TCIMACRO{\U{2115} }%
%BeginExpansion
\mathbb{N}
%EndExpansion
$.

\begin{proposition}
\label{kjn} Consider a random dynamical system $\mathcal{R}$ $=$
$(\Omega,\mathcal{F},\mathbb{P},\sigma,\mathcal{X},\mathcal{L})$ with an
ergodic base $\sigma$. Then, for each $n\in%
%TCIMACRO{\U{2115} }%
%BeginExpansion
\mathbb{N}
%EndExpansion
$, $\mathcal{R}^{(n)}$ $=$ $(\Omega,\mathcal{F},\mathbb{P},\sigma
^{n},\mathcal{X},\mathcal{L}^{(n)})$ is a random dynamical system, with a
possibly non-ergodic base $\sigma^{n}$. For each $\omega\in\Omega$, let
\begin{align*}
\lambda_{n}(\omega)  &  =\lim_{k\rightarrow\infty}\frac{1}{k}\log
\Vert\mathcal{L}_{\omega}^{(nk)}\Vert\text{,}\\
\mathcal{K}_{n}(\omega)  &  =\lim_{k\rightarrow\infty}\frac{1}{k}\log
\Vert\mathcal{L}_{\omega}^{(nk)}\Vert_{ic(\mathcal{X})}\text{.}%
\end{align*}
Then%
\begin{align*}
\lambda_{n}(\omega)  &  =n\lambda^{\ast}(\mathcal{R})\text{,}\\
\mathcal{K}_{n}(\omega)  &  =n\mathcal{K}^{\ast}(\mathcal{R})\text{,}%
\end{align*}
$\mathbb{P}-$almost everywhere.
\end{proposition}

\begin{proof}
For each $n\in%
%TCIMACRO{\U{2115} }%
%BeginExpansion
\mathbb{N}
%EndExpansion
$, $\{\mathcal{L}_{\omega}^{(nk)}\}_{k=1}^{\infty}$ is a subsequence of
$\{\mathcal{L}_{\omega}^{(k)}\}_{k=1}^{\infty}$. Thus, the proof follows from
Remark~\ref{re4}.
\end{proof}

While the transformation $\sigma^{n}$ may be non-ergodic when $\sigma$ is
ergodic, the following result ensures that $\sigma^{n}$ is ergodic on some
subset $Z\subset\Omega$. This result will be used in the proof of Proposition
\ref{singmanergodic copy(3)}.

\begin{lemma}
[{Gonz\'alez-Tokman \& Quas \cite[Lemma $35$]{CQ1}}]\label{singmanergodic}Let
$\sigma$ be an ergodic $\mathbb{P}$-preserving transformation of
$(\Omega,\mathcal{F},\mathbb{P})$ and let $n\in\mathbb{N}$. Then there exists
$k$, a factor of $n$, and a $\sigma^{n}$-invariant subset $Z$ of $\Omega$ of
measure $1/k$ such that $\Omega=%
%TCIMACRO{\dbigcup \limits_{s=0}^{k-1}}%
%BeginExpansion
{\displaystyle\bigcup\limits_{s=0}^{k-1}}
%EndExpansion
\sigma^{-\ell}Z$ and $\sigma^{n}|_{Z}$ is ergodic. When $\sigma$ is
invertible, this argument also applies to $n<0$.
\end{lemma}

\subsection{Admissible random Jab\l o\'{n}ski maps and quasi-compactness}

\begin{definition}
\label{repa} A partition $\mathcal{B}=\{B_{1},\dots,B_{q}\}$ of $I^{n}$ is
called \emph{rectangular }if for each $j=1,\dots,q$,
\[
B_{j}=%
%TCIMACRO{\dprod \limits_{i=1}^{n}}%
%BeginExpansion
{\displaystyle\prod\limits_{i=1}^{n}}
%EndExpansion
B_{ij}\text{,}%
\]
where $B_{ij}=[a_{ij},b_{ij})$ if $b_{ij}<1$ and $B_{ij}=[a_{ij},b_{ij}]$ if
$b_{ij}=1$.
\end{definition}

A piecewise map \ $f:I^{n}\circlearrowleft$ defined on the rectangular
partition given in Definition \ref{repa} is generally written as%
\[
f(x_{1},\dots,x_{n})=(\varphi_{1,j}(x_{1},\dots,x_{n}),\dots,\varphi
_{n,j}(x_{1},\dots,x_{n}))\text{,}%
\]
where $(x_{1},\dots,x_{n})\in B_{j}$, $j=1,\dots,q$. Following \cite{J}, we
next introduce Jab\l o\'{n}ski maps as a special case of such maps. We then
define random Jab\l o\'{n}ski maps and admissible random Jab\l o\'{n}ski maps.

\begin{definition}
[\emph{Jab\l o\'{n}ski }\cite{J}]\label{def:JabMap} A map $f:I^{n}%
\circlearrowleft$ is called a \emph{Jab\l o\'{n}ski map} if it is piecewise
defined on a rectangular partition $\mathcal{B}=\{B_{1},\dots,B_{q}\}$ of
$I^{n}$ and is given by the formula%
\[
f(x_{1},\dots,x_{n})=(\varphi_{1,j}(x_{1}),\dots,\varphi_{n,j}(x_{n}))\text{,}%
\]
where $(x_{1},\dots,x_{n})\in B_{j},$ $j=1,2,\dots,q$. The vertices of the
rectangles in $\mathcal{B}$ which lie in the interior of $I^{n}$ are called
the \emph{crossing points} of $f$. The real valued maps $\varphi_{i,j}%
:B_{ij}\rightarrow\lbrack0,1]$ are called the \emph{components} of $f$. We use
$\mathcal{J}$ to denote for the class of Jab\l o\'{n}ski maps on $I^{n}$.
\end{definition}

While the above family of Jab\l o\'{n}ski maps may seem restrictive, in
\cite{BLG}, Boyarsky, G\'{o}ra and Lou proved that for any piecewise $C^{2}$
map $f$ defined on a rectangular partition of $I^{n}$, $f$ can be approximated
by a sequence of piecewise $C^{2}$ Jab\l o\'{n}ski transformations. In other
words, there exists a sequence of Jab\l o\'{n}ski maps $f_{n}$ that converges
pointwise to $f.$ Moreover, the corresponding sequence of invariant densities
of $f_{n}$ (which exists by \cite{J}) converges weakly to an invariant density
of $f$. Generally speaking, Jab\l o\'{n}ski maps can be seen as the basis maps
for a much larger class of piecewise defined maps on the $n$ dimensional rectangle.

\begin{definition}
Let $(\Omega,\mathcal{F},\mathbb{P})$ be a probability space and
$\sigma:\Omega\circlearrowleft$ an invertible, ergodic and $\mathbb{P-}%
$preserving transformation. A \emph{random Jab\l o\'{n}ski map }%
$\mathscr{F}$\emph{ over }$\sigma$ is a map $\mathscr{F}:\Omega\rightarrow
\mathcal{J}$, where $f_{\omega}:=\mathscr{F(\omega
)}:I^{n}\circlearrowleft$. Hence, for each $\omega\in\Omega$, there exists a
rectangular partition $\mathcal{B}^{\omega}$ of $I^{n}$, say,
\[
\mathcal{B}^{\omega}=\{B_{1}^{\omega},\dots,B_{q_{\omega}}^{\omega}\}\text{,}%
\]
where $q_{\omega}$ is a positive integer. If $x=(x_{1},\dots,x_{n})\in
B_{j}^{\omega}$, where $j\in\{1,\dots,q_{\omega}\}$, then we have%
\[
f_{\omega}(x)=(\varphi_{\omega,1,j}(x_{1}),\dots,\varphi_{\omega,n,j}%
(x_{n}))\text{,}%
\]
where $B_{j}^{\omega}=%
%TCIMACRO{\dprod \limits_{i=1}^{n}}%
%BeginExpansion
{\displaystyle\prod\limits_{i=1}^{n}}
%EndExpansion
[a_{i}^{\omega,j},b_{i}^{\omega,j})$ and $\varphi_{\omega,i,j}$ is a map from
$[a_{i}^{\omega,j},b_{i}^{\omega,j}]$ into $[0,1]$. For $k\in%
%TCIMACRO{\U{2115} }%
%BeginExpansion
\mathbb{N}
%EndExpansion
$, the $k$ fold composition $f_{\omega}^{(k)}$ is defined as\
\begin{equation}
f_{\omega}^{(k)}:=f_{\sigma^{k-1}\omega}\circ\dots\circ f_{\sigma\omega}\circ
f_{\omega}\text{.}%
\end{equation}

\end{definition}

For simplicity, we sometimes refer to the range of $\mathscr{F}$, that is
$\{f_{\omega}\}_{\omega\in\Omega}$, as the random Jab\l o\'{n}ski map. A
random Jab\l o\'{n}ski map gives rise to a random dynamical system, where
$\mathcal{X}=BV(I^{n})$ and $\mathcal{L}_{\omega}=\mathcal{L}_{f_{\omega}}$.
The next proposition proves that for all $\omega\in\Omega$, $\mathcal{L}%
_{\omega}$ is a bounded operator on $BV(I^{n})$.

\begin{proposition}
\label{BVbded}For all $\omega\in\Omega$, $\mathcal{L}_{\omega}$ is a bounded
operator on $BV(I^{n})$.
\end{proposition}

\begin{proof}
Let $\omega\in\Omega$,
\begin{align*}
\Vert\mathcal{L}_{\omega}h\Vert_{BV}  &  =%
%TCIMACRO{\dint \limits_{I^{n}}}%
%BeginExpansion
{\displaystyle\int\limits_{I^{n}}}
%EndExpansion
|\mathcal{L}_{\omega}h|dm+\overset{I^{n}}{\mathbb{V}}\mathcal{L}_{\omega}h\\
&  \leq%
%TCIMACRO{\dint \limits_{I^{n}}}%
%BeginExpansion
{\displaystyle\int\limits_{I^{n}}}
%EndExpansion
\mathcal{L}_{\omega}|h|dm+\overset{I^{n}}{\mathbb{V}}\mathcal{L}_{\omega}h\\
&  =%
%TCIMACRO{\dint \limits_{I^{n}}}%
%BeginExpansion
{\displaystyle\int\limits_{I^{n}}}
%EndExpansion
|h|dm+\overset{I^{n}}{\mathbb{V}}\mathcal{L}_{\omega}h\\
&  =\Vert h\Vert_{1}+\overset{I^{n}}{\mathbb{V}}\mathcal{L}_{\omega}h
\end{align*}
By the Lasota-Yorke inequality provided in the proof of Theorem $1$ in
\cite{J}\textbf{,} the last term is less than or equal%
\[
\Vert h\Vert_{1}+\alpha\Vert h\Vert_{1}+\beta\overset{A}{\mathbb{V}}h\leq
\max(1+\alpha,\beta)\Vert h\Vert_{BV}\text{.}%
\]
for some $\alpha,\beta>0$.
\end{proof}

We note that
\begin{equation}
\mathcal{L}_{\omega}^{(k)}=\mathcal{L}_{f_{\sigma^{k-1}\omega}\circ\dots\circ
f_{\sigma\omega}\circ f_{\omega}}=\mathcal{L}_{\sigma^{k-1}\omega}\circ
\dots\circ\mathcal{L}_{\sigma\omega}\circ\mathcal{L}_{\omega}\text{.}
\label{toc}%
\end{equation}
In what follows, we will use the notation $\overset{\rightarrow}{r}%
=(r_{1},\dots,r_{n})\in%
%TCIMACRO{\U{2115} }%
%BeginExpansion
\mathbb{N}
%EndExpansion
^{n}$ and denote the set of vector indices by
\[%
%TCIMACRO{\U{2124} }%
%BeginExpansion
\mathbb{Z}
%EndExpansion
_{_{\overset{\rightarrow}{r}}}:=\{\overset{\rightarrow}{s}=(s_{1},s_{2}%
,\dots,s_{n}):\text{ }1\leq s_{i}\leq r_{i}\}\text{.}%
\]

\begin{remark}
\label{re1} If the random Jab\l o\'{n}ski map $\mathscr{F}$ has a finite
range, then for each $k_{0}\in%
%TCIMACRO{\U{2115} }%
%BeginExpansion
\mathbb{N}
%EndExpansion
$, there exists a common partition $\mathcal{B}=\mathcal{B}(k_{0})$ of $I^{n}$
into maximal rectangles such that the components of the maps $\{f_{\omega
}^{(k_{0})}\}_{\omega\in\Omega}$ are $C^{2}$ and monotonic on their interval domains.
\end{remark}

\begin{remark}
\label{re2} The generator $\mathcal{L}:\Omega\mathcal{\rightarrow
}L(\mathcal{X},\mathcal{X})$ of the random dynamical system generated by a
random Jab\l o\'{n}ski map $\mathscr{F}$ is $\mathbb{P}$-continuous if its
range is at most countably infinite (consisting of, say, $f_{1},f_{2},\dots$)
and the preimage of each $f_{j}$ is a measurable set. Our results will be
valid when there exists a common partition $\mathcal{B}$ of $I^{n}$ into
rectangles such that the components of the maps $\{f_{\omega}^{(N)}%
\}_{\omega\in\Omega}$ are $C^{2}$ and monotonic on their interval domains,
where $N\in%
%TCIMACRO{\U{2115} }%
%BeginExpansion
\mathbb{N}
%EndExpansion
$ satisfies the condition given in \eqref{Nc}. In this case, for each
$i=1,\dots,n$, there exists a partition
\[
0=a_{i,0}<a_{i,1}<\dots<a_{i,r_{i}}=1\text{,}%
\]
for some $r_{i}\in%
%TCIMACRO{\U{2115} }%
%BeginExpansion
\mathbb{N}
%EndExpansion
$. Let $B_{s_{i}}=$ $[a_{i,s_{i}-1},a_{i,s_{i}})$ when $s_{i}=1,2,\dots
,r_{i}-1$ and $B_{r_{i}}=$ $[a_{i,r_{i}-1},a_{i,r_{i}}]$. For each vector
index $\overset{\rightarrow}{s}\in%
%TCIMACRO{\U{2124} }%
%BeginExpansion
\mathbb{Z}
%EndExpansion
_{_{\overset{\rightarrow}{r}}}$, we denote the $n$ dimensional rectangle by
$B_{\overset{\rightarrow}{s}}=%
%TCIMACRO{\dprod \limits_{i=1}^{n}}%
%BeginExpansion
{\displaystyle\prod\limits_{i=1}^{n}}
%EndExpansion
B_{s_{i}}$. The common rectangular partition is given by%
\[
\mathcal{B}=\{B_{\overset{\rightarrow}{s}}:\text{ }\overset{\rightarrow}{s}\in%
%TCIMACRO{\U{2124} }%
%BeginExpansion
\mathbb{Z}
%EndExpansion
_{_{\overset{\rightarrow}{r}}}\}\text{.}%
\]
For each $\omega\in\Omega$ and $\overset{\rightarrow}{s}\in%
%TCIMACRO{\U{2124} }%
%BeginExpansion
\mathbb{Z}
%EndExpansion
_{_{\overset{\rightarrow}{r}}}$, we write the map $f_{\omega}$ with respect to
$\mathcal{B}$ as%
\[
f_{\omega}(x)=(\varphi_{\omega,1,\overset{\rightarrow}{s}}(x_{1}%
),\dots,\varphi_{\omega,n,\overset{\rightarrow}{s}}(x_{n}))\text{, }%
x=(x_{1},\dots,x_{n})\in B_{\overset{\rightarrow}{s}}\text{,}%
\]
and, for each $k\in%
%TCIMACRO{\U{2115} }%
%BeginExpansion
\mathbb{N}
%EndExpansion
$, the map $f_{\omega}^{(k)}$ as%
\[
f_{\omega}^{(k)}(x)=(\varphi_{\omega,1,k,\overset{\rightarrow}{s}}%
(x_{1}),\dots,\varphi_{\omega,n,k,\overset{\rightarrow}{s}}(x_{n}%
))\text{,\ }x=(x_{1},\dots,x_{n})\in B_{\overset{\rightarrow}{s}}\text{.}%
\]

\end{remark}

\begin{remark}
\label{re3} One can associate to the random Jab\l o\'{n}ski map
$\mathscr{F}=\{f_{\omega}\}_{\omega\in\Omega}$, the skew product map $F$ on
$\Omega\times I^{n}$ which encodes the dynamics of the whole system%
\begin{equation}
F(\omega,x)=(\sigma\omega,f_{\omega}(x))\text{.} \label{sp}%
\end{equation}

\end{remark}

%Let $(\Omega,\mathcal{F},\mathbb{P)}$ be a probability space and
%$(\mathcal{X},\Vert$ $\Vert)$ a Banach space. A random dynamical system is a
%tuple $\mathcal R=(\Omega,\mathcal{F},\mathbb{P},\sigma,\mathcal{X},\mathcal{L}%
%\mathbb{)}$, where $\sigma$ is an invertible measure-preserving transformation
%of $(\Omega,\mathcal{F},\mathbb{P)}$ and $\mathcal{L}:\Omega
%\mathcal{\rightarrow}L(\mathcal{X},\mathcal{X})$ is a family of bounded linear
%maps of $\mathcal{X}$, called the generator. We let $\mathcal L_\omega = \mathcal L(\omega)$. A random dynamical system defines
%a cocycle
%$(k,\omega)\mapsto \mathcal{L}_{\omega}^{(k)} =\mathcal{L}_{\sigma^{k-1}\omega}\circ\dots\circ\mathcal{L}_{\sigma\omega}\circ\mathcal{L}_{\omega}$.

%\begin{equation}
%(k,\omega)\mapsto \mathcal{L}_{\omega}^{(k)} :=\mathcal{L}_{f_{\omega}^{(k)}%
%}=\mathcal{L}_{f_{\sigma^{k-1}\omega}\circ\dots \circ f_{\sigma\omega}\circ
%f_{\omega}}=\mathcal{L}_{\sigma^{k-1}\omega}\circ\dots \circ\mathcal{L}%
%_{\sigma\omega}\circ\mathcal{L}_{\omega}\text{,}%
%\end{equation}
%where $\mathcal{L}_{\omega}=\mathcal{L}_{f_{\omega}}$.

Expanding properties for dynamical systems lead to chaotic behavior of the
orbits. However, they usually give rise to good ergodic properties like the
existence of absolutely continuous invariant measures. Next we introduce the
admissible random Jab\l o\'{n}ski maps. This definition involves a formulation
of an expanding-on-average condition.

\begin{definition}
\label{arjm} Using the notation in Remark \ref{re2}, a random Jab\l o\'{n}ski
map $\mathscr{F}$ is called \emph{admissible} if all the components
$\varphi_{\omega,i,\overset{\rightarrow}{s}}$ are $C^{2}$ and monotonic on
$[a_{i,s_{i}-1},a_{i,s_{i}}]$ and there exists a constant $\gamma>0$ such
that
%for every $i=1,2,\dots,n,$ we have%
\begin{equation}
\Gamma:=\int_{\Omega}\underset{i=1,\dots,n}{\min} \log(\mathbb{\gamma}%
_{i}(\omega))d\mathbb{P(\omega)>\gamma}\text{,} \label{ec}%
\end{equation}
%\begin{equation}
%\Gamma_{i}:=\int_{\Omega}\log(\mathbb{\gamma}_{i}(\omega))d\mathbb{P(\omega
%)>\gamma}\text{,} \label{ec}%
%\end{equation}
where%
\begin{equation}
\mathbb{\gamma}_{i}(\omega):=\inf_{_{\substack{\overset{\rightarrow}{s}\in%
%TCIMACRO{\U{2124} }%
%BeginExpansion
\mathbb{Z}
%EndExpansion
_{_{\overset{\rightarrow}{r}}}\\x_{i}\in\lbrack a_{i,s_{i}-1},a_{i,s_{i}}]}%
}}(|\varphi_{\omega,i,\overset{\rightarrow}{s}}^{\prime}(x_{i})|)\text{.}
\label{fec}%
\end{equation}
In addition, we assume the mapping $\omega\mapsto\mathcal{L}_{\omega}$ is
$\mathbb{P}$-continuous.
\end{definition}

\section{Random Lasota-Yorke inequality and quasi-compactness}

\label{S:3}

In the next theorem, we establish a suitable Lasota-Yorke inequality on the
space of bounded variation $BV(I^{n})$ and we use it to prove the
quasi-compactness property for admissible random Jab\l o\'{n}ski maps.

\begin{theorem}
\label{quco} Let $\mathscr{F}=\{f_{\omega}\}_{\omega\in\Omega}$ be an
admissible random Jab\l o\'{n}ski map. Then:

(i) the random dynamical system generated by $\mathscr{F}$ is quasi-compact; and

(ii) its maximal Lyapunov exponent $\lambda^{\ast}$ is zero.
\end{theorem}

\begin{proof}
[Proof of Theorem \ref{quco} (i)]The first step is to show that there are
$N\in%
%TCIMACRO{\U{2115} }%
%BeginExpansion
\mathbb{N}
%EndExpansion
$ and positive measurable functions $\alpha_{1},\alpha_{2}:\Omega\rightarrow%
%TCIMACRO{\U{211d} }%
%BeginExpansion
\mathbb{R}
%EndExpansion
^{+}$such that $\int_{\Omega}\log\alpha_{1}(\omega)d\mathbb{P(\omega)<}0$ and%
\begin{equation}
\overset{I^{n}}{\mathbb{V}}\mathcal{L}_{\omega}^{(N)}h\leq\alpha_{1}%
(\omega)\overset{I^{n}}{\mathbb{V}}h+\alpha_{2}(\omega)\Vert h\Vert
_{1}\text{,}\label{lyi}%
\end{equation}
for all $h$\textbf{$\in$}$BV(I^{n})$, where $\mathcal{L}_{\omega}^{(N)}$ is
defined in \eqref{toc}. Let $x=(x_{1},\dots,x_{n})\in I^{n}$ and $\omega
\in\Omega$, choose $N\in%
%TCIMACRO{\U{2115} }%
%BeginExpansion
\mathbb{N}
%EndExpansion
$ such that
\begin{equation}
N\gamma>\log(3)\text{,}\label{Nc}%
\end{equation}
where $\gamma$ satisfies the condition \eqref{ec}. Let $\overset{\rightarrow
}{s_{0}}=\overset{\rightarrow}{s}$ be the label of the unique rectangle in
$\mathcal{B}$ for which $x\in B_{\overset{\rightarrow}{s}}$ and for
$k=1,2,\dots$, let $\overset{\rightarrow}{s_{k}}\in%
%TCIMACRO{\U{2124} }%
%BeginExpansion
\mathbb{Z}
%EndExpansion
_{\overset{\rightarrow}{r}}$ be such that
\begin{align*}
f_{\omega}^{(k)}(x) &  =(\varphi_{\omega,1,k,\overset{\rightarrow}{s}}%
(x_{1}),\dots,\varphi_{\omega,n,k,\overset{\rightarrow}{s}}(x_{n}))\\
&  =\Big(\varphi_{\sigma^{k-1}\omega,1,\overset{\rightarrow}{s_{k-1}}}%
\circ\dots\circ\varphi_{\omega,1,\overset{\rightarrow}{s}}(x_{1}%
),\dots,\varphi_{\sigma^{k-1}\omega,n,\overset{\rightarrow}{s_{k-1}}}%
\circ\dots\circ\varphi_{\omega,n,\overset{\rightarrow}{s}}(x_{n})\Big)\\
&  \in B_{\overset{\rightarrow}{s_{k}}}\text{.}%
\end{align*}
Note that for any $i=1,2,\dots,n$,%
\begin{align}
&  \int_{\Omega}\inf_{_{\substack{\overset{\rightarrow}{s}\in%
%TCIMACRO{\U{2124} }%
%BeginExpansion
\mathbb{Z}
%EndExpansion
_{\overset{\rightarrow}{r}}\\x_{i}\in\lbrack a_{i,s_{i}-1},a_{i,s_{i}}]}}}%
\log(|\varphi_{\omega,i,N,\overset{\rightarrow}{s}}^{\prime}(x_{i}%
)|)d\mathbb{P(\omega)}\label{lolololo}\\
&  =\int_{\Omega}\inf_{_{\substack{\overset{\rightarrow}{s}\in%
%TCIMACRO{\U{2124} }%
%BeginExpansion
\mathbb{Z}
%EndExpansion
_{\overset{\rightarrow}{r}}\\x_{i}\in\lbrack a_{i,s_{i}-1},a_{i,s_{i}}]}}}%
\log\Big(|(\varphi_{\sigma^{N-1}\omega,i,\overset{\rightarrow}{s_{N-1}}}%
\circ\dots\circ\varphi_{\sigma\omega,i,\overset{\rightarrow}{s_{1}}}%
\circ\varphi_{\omega,i,\overset{\rightarrow}{s}})^{\prime}(x_{i}%
)|\Big)d\mathbb{P(\omega)}\nonumber\\
&  \mathbb{\geq}%
%TCIMACRO{\dsum \limits_{k=0}^{N-1}}%
%BeginExpansion
{\displaystyle\sum\limits_{k=0}^{N-1}}
%EndExpansion
\int_{\Omega}\inf_{_{\substack{\overset{\rightarrow}{s}\in%
%TCIMACRO{\U{2124} }%
%BeginExpansion
\mathbb{Z}
%EndExpansion
_{\overset{\rightarrow}{r}}\\x_{i}\in\lbrack a_{i,s_{i}-1},a_{i,s_{i}}]}}}%
\log\Big(|\varphi_{\sigma^{k}(\omega),i,\overset{\rightarrow}{s_{k}}}^{\prime
}(\varphi_{\omega,i,k-1,\overset{\rightarrow}{s_{k-1}}}(x_{i}%
))|\Big)d\mathbb{P(\omega)}\nonumber\\
&  \geq%
%TCIMACRO{\dsum \limits_{k=0}^{N-1}}%
%BeginExpansion
{\displaystyle\sum\limits_{k=0}^{N-1}}
%EndExpansion
\int_{\Omega}\log(\mathbb{\gamma}_{i}(\omega))d\mathbb{P(\omega)\geq}%
N\gamma>\log(3)\text{.}\nonumber
\end{align}

Let $\mathcal{E}$ be the set of functions of the form $g=%
%TCIMACRO{\dsum \limits_{j=1}^{M}}%
%BeginExpansion
{\displaystyle\sum\limits_{j=1}^{M}}
%EndExpansion
g_{j}\mathcal{X}_{A_{j}}\text{,}$ where $A_{j}=%
%TCIMACRO{\dprod \limits_{i=1}^{n}}%
%BeginExpansion
{\displaystyle\prod\limits_{i=1}^{n}}
%EndExpansion
[\alpha_{i}^{j},\beta_{i}^{j}]\subseteq I^{n}$ and $g_{j}$:$I^{n}\rightarrow%
%TCIMACRO{\U{211d} }%
%BeginExpansion
\mathbb{R}
%EndExpansion
$ is a $C^{1}$ function on $A_{j}$, By \cite[Remark $1$]{J}, $\mathcal{E}$
forms a dense subset of the space $L^{1}(I^{n})$. By \cite[Remark $5$]{J},
$\mathcal{E}\subset BV(I^{n})$.

We argue in a similar way to the proof of Theorem $1$ in \cite{J}. We provide
the Lasota-Yorke inequality on elements of $\mathcal{E}$. Since the BV norm is
a continuous function and by Proposition \ref{BVbded} the transfer operator is
bounded, using a density argument, the inequality can be extended to elements
of $BV(I^{n})$.

Let $h\in\mathcal{E}$ be such that $h\geq0$ and for any $i=1,\dots,n$, let
$h_{i}\in\mathcal{E}$ be such that $h_{i}=h$ Lebesgue almost everywhere with
the property%
\[
\int_{P_{i}(I^{n})}\underset{i}{\overset{I^{n}}{\mathbf{V}}}h_{i}%
dm=\underset{i}{\overset{I^{n}}{\mathbb{V}}}h\text{.}%
\]
Let
\begin{align*}
\Psi_{\omega,i,N,\overset{\rightarrow}{s}}  &  :=\varphi_{\omega
,i,N,\overset{\rightarrow}{s}}^{-1}\text{,}\\
\delta_{\omega,i,N,\overset{\rightarrow}{s}}  &  :=|\Psi_{\omega
,i,N,\overset{\rightarrow}{s}}^{\prime}|\text{,}\\
I_{\omega,N,\overset{\rightarrow}{s}}  &  :=%
%TCIMACRO{\dprod \limits_{i=1}^{n}}%
%BeginExpansion
{\displaystyle\prod\limits_{i=1}^{n}}
%EndExpansion
\varphi_{\omega,i,N,\overset{\rightarrow}{s}}([a_{i,s_{i}-1},a_{i,s_{i}%
}])\text{.}%
\end{align*}
The transfer operator $\mathcal{L}_{\omega}^{(N)}$ applied to $h$ evaluated at
$x=(x_{1},\dots,x_{n})\in I^{n}$ is given by%
\[
\mathcal{L}_{\omega}^{(N)}h(x)=%
%TCIMACRO{\dsum \limits_{\overset{\rightarrow}{s}\in\mathbb{Z}_{_{\overset
%{\rightarrow}{r}}}}}%
%BeginExpansion
{\displaystyle\sum\limits_{\overset{\rightarrow}{s}\in\mathbb{Z}%
_{_{\overset{\rightarrow}{r}}}}}
%EndExpansion
h\Big(\Psi_{\omega,1,N,\overset{\rightarrow}{s}}(x_{1}),\dots,\Psi
_{\omega,n,N,\overset{\rightarrow}{s}}(x_{n})\Big)%
%TCIMACRO{\dprod \limits_{j=1}^{n}}%
%BeginExpansion
{\displaystyle\prod\limits_{j=1}^{n}}
%EndExpansion
\delta_{\omega,j,N,\overset{\rightarrow}{s}}(x_{j})1_{I_{\omega,N,\overset
{\rightarrow}{s}}}(x)\text{.}%
\]
If we apply$\underset{i}{\overset{I^{n}}{\mathbf{V}}}$ for the the
$L^{1}(I^{n})$-function $\mathcal{L}_{\omega}^{(N)}h_{i}$ and the take the
integral over $P_{i}(I^{n})$, we get%
\begin{equation}
\int_{P_{i}(I^{n})}\underset{i}{\overset{I^{n}}{\mathbf{V}}}\mathcal{L}%
_{\omega}^{(N)}h_{i}dm\leq I_{1}+I_{2}\text{,} \label{bound}%
\end{equation}
where%
\begin{align*}
I_{1}  &  =%
%TCIMACRO{\dsum \limits_{\overset{\rightarrow}{s}\in\mathbb{Z}_{_{\overset
%{\rightarrow}{r}}}}}%
%BeginExpansion
{\displaystyle\sum\limits_{\overset{\rightarrow}{s}\in\mathbb{Z}%
_{_{\overset{\rightarrow}{r}}}}}
%EndExpansion
\int_{P_{i}(I_{\omega,N,\overset{\rightarrow}{s}})}\underset{i}{\overset
{I_{\omega,N,\overset{\rightarrow}{s}}}{\mathbf{V}}}h_{i}\Big(\Psi
_{\omega,1,N,\overset{\rightarrow}{s}}(x_{1}),\dots,\Psi_{\omega
,n,N,\overset{\rightarrow}{s}}(x_{n})\Big)%
%TCIMACRO{\dprod \limits_{j=1}^{n}}%
%BeginExpansion
{\displaystyle\prod\limits_{j=1}^{n}}
%EndExpansion
\delta_{\omega,j,N,\overset{\rightarrow}{s}}(x_{j})dm\text{,}\\
I_{2}  &  =%
%TCIMACRO{\dsum \limits_{\overset{\rightarrow}{s}\in\mathbb{Z}_{_{\overset
%{\rightarrow}{r}}}}}%
%BeginExpansion
{\displaystyle\sum\limits_{\overset{\rightarrow}{s}\in\mathbb{Z}%
_{_{\overset{\rightarrow}{r}}}}}
%EndExpansion
\int_{P_{i}(I_{\omega,N,\overset{\rightarrow}{s}})}\Big(|h_{i}(\Psi
_{\omega,1,N,\overset{\rightarrow}{s}}(x_{1}),\dots,\Psi_{\omega
,n,N,\overset{\rightarrow}{s}}(x_{n}))|\delta_{\omega,i,N,\overset
{\rightarrow}{s}}(\varphi_{\omega,i,N,\overset{\rightarrow}{s}}(a_{i,s_{i}%
}))\\
&  +|h_{i}(\Psi_{\omega,1,N,\overset{\rightarrow}{s}}(x_{1}),\dots
,\Psi_{\omega,n,N,\overset{\rightarrow}{s}}(x_{n}))|\delta_{\omega
,i,N,\overset{\rightarrow}{s}}(\varphi_{\omega,i,N,\overset{\rightarrow}{s}%
}(a_{i,s_{i}-1}))\Big)%
%TCIMACRO{\dprod \limits_{\overset{j=1}{j\neq i}}^{n}}%
%BeginExpansion
{\displaystyle\prod\limits_{\overset{j=1}{j\neq i}}^{n}}
%EndExpansion
\delta_{j}^{\omega,N,\overset{\rightarrow}{s}}(x_{j})dm\text{.}%
\end{align*}
Let%
\begin{equation}
\rho_{\omega,i,N}=\underset{\overset{\rightarrow}{s}\in%
%TCIMACRO{\U{2124} }%
%BeginExpansion
\mathbb{Z}
%EndExpansion
_{\overset{\rightarrow}{r}}}{\sup}\delta_{\omega,i,N,\overset{\rightarrow}{s}%
}\text{,} \label{yaraaa}%
\end{equation}

and
\[
K_{\omega,i,N}=\frac{\underset{\overset{\rightarrow}{s}\in%
%TCIMACRO{\U{2124} }%
%BeginExpansion
\mathbb{Z}
%EndExpansion
_{\overset{\rightarrow}{r}}}{\sup}\delta_{\omega,i,N,\overset{\rightarrow}{s}%
}^{\prime}}{\underset{\overset{\rightarrow}{s}\in%
%TCIMACRO{\U{2124} }%
%BeginExpansion
\mathbb{Z}
%EndExpansion
_{\overset{\rightarrow}{r}}}{\inf}\delta_{\omega,i,N,\overset{\rightarrow}{s}%
}}+\underset{\overset{\rightarrow}{s}\in%
%TCIMACRO{\U{2124} }%
%BeginExpansion
\mathbb{Z}
%EndExpansion
_{\overset{\rightarrow}{r}}}{\sup}\delta_{\omega,i,N,\overset{\rightarrow}{s}%
}\text{.}%
\]

These constants are motivated from the ones given in Theorem $1$ in \cite{P}
also Theorem $2$ in \cite{BL}.\ By Inequality $7$ in \cite{P} adapted to our
notation, we have%

\begin{align*}
I_{1}  &  \leq2\rho_{\omega,i,N}%
%TCIMACRO{\dsum \limits_{\overset{\rightarrow}{s}\in\mathbb{Z}_{P_{i}%
%(\overset{\rightarrow}{r})}}}%
%BeginExpansion
{\displaystyle\sum\limits_{\overset{\rightarrow}{s}\in\mathbb{Z}%
_{P_{i}(\overset{\rightarrow}{r})}}}
%EndExpansion
\int_{P_{i}(I_{\omega,N,\overset{\rightarrow}{s}})}\underset{i}{\overset
{P_{i}(I_{\omega,N,\overset{\rightarrow}{s}})\times I}{\mathbf{V}}}%
h_{i}\Big(\Psi_{\omega,1,N,\overset{\rightarrow}{s}}(x_{1}),\dots,x_{i}%
,\dots,\Psi_{\omega,n,N,\overset{\rightarrow}{s}}(x_{n})\Big)\\
&
%TCIMACRO{\dprod \limits_{\overset{j=1}{j\neq i}}^{n}}%
%BeginExpansion
{\displaystyle\prod\limits_{\overset{j=1}{j\neq i}}^{n}}
%EndExpansion
\delta_{\omega,j,N,\overset{\rightarrow}{s}}(x_{j})dm\\
&  +K_{\omega,i,N}%
%TCIMACRO{\dsum \limits_{\overset{\rightarrow}{s}\in\mathbb{Z}_{P_{i}%
%(\overset{\rightarrow}{r})}}}%
%BeginExpansion
{\displaystyle\sum\limits_{\overset{\rightarrow}{s}\in\mathbb{Z}%
_{P_{i}(\overset{\rightarrow}{r})}}}
%EndExpansion
\int_{P_{i}(I_{\omega,N,\overset{\rightarrow}{s}})}%
%TCIMACRO{\dint \limits_{0}^{1}}%
%BeginExpansion
{\displaystyle\int\limits_{0}^{1}}
%EndExpansion
h_{i}\Big(\Psi_{\omega,1,N,\overset{\rightarrow}{s}}(x_{1}),\dots,x_{i}%
,\dots,\Psi_{\omega,n,N,\overset{\rightarrow}{s}}(x_{n})\Big)dx_{i}\\
&
%TCIMACRO{\dprod \limits_{\overset{j=1}{j\neq i}}^{n}}%
%BeginExpansion
{\displaystyle\prod\limits_{\overset{j=1}{j\neq i}}^{n}}
%EndExpansion
\delta_{\omega,j,N,\overset{\rightarrow}{s}}(x_{j})dm\text{.}%
\end{align*}

Making the change of variables $U=\Psi^{-1}$, as in Lemma $3$ in \cite{P},
then the above sum is equal to%

\begin{align*}
&  2\rho_{\omega,i,N}%
%TCIMACRO{\dsum \limits_{\overset{\rightarrow}{s}\in\mathbb{Z}_{P_{i}%
%(\overset{\rightarrow}{r})}}}%
%BeginExpansion
{\displaystyle\sum\limits_{\overset{\rightarrow}{s}\in\mathbb{Z}%
_{P_{i}(\overset{\rightarrow}{r})}}}
%EndExpansion
\int_{P_{i}(B_{\overset{\rightarrow}{s}})}\underset{i}{\overset{P_{i}%
(B_{\overset{\rightarrow}{s}})\times I}{\mathbf{V}}}h_{i}(x_{1},\dots
,x_{n})dm\\
&  +K_{\omega,i,N}%
%TCIMACRO{\dsum \limits_{\overset{\rightarrow}{s}\in\mathbb{Z}_{P_{i}%
%(\overset{\rightarrow}{r})}}}%
%BeginExpansion
{\displaystyle\sum\limits_{\overset{\rightarrow}{s}\in\mathbb{Z}%
_{P_{i}(\overset{\rightarrow}{r})}}}
%EndExpansion
\int_{P_{i}(B_{\overset{\rightarrow}{s}})}%
%TCIMACRO{\dint \limits_{0}^{1}}%
%BeginExpansion
{\displaystyle\int\limits_{0}^{1}}
%EndExpansion
h_{i}(x_{1},\dots,x_{n})dx_{i}dm\text{.}%
\end{align*}

Since $\{B_{\overset{\rightarrow}{s}}:$ $\overset{\rightarrow}{s}\in%
%TCIMACRO{\U{2124} }%
%BeginExpansion
\mathbb{Z}
%EndExpansion
_{\overset{\rightarrow}{r}}\}$ forms a partition for $I^{n},$ the last sum is
equal to%

\begin{align}
&  2\rho_{\omega,i,N}\int_{P_{i}(I^{n})}\underset{i}{\overset{I^{n}%
}{\mathbf{V}}}h_{i}dm+K_{\omega,i,N}\int_{P_{i}(I^{n})}(%
%TCIMACRO{\dint \limits_{0}^{1}}%
%BeginExpansion
{\displaystyle\int\limits_{0}^{1}}
%EndExpansion
h_{i}dx_{i})dm\nonumber\\
&  \leq2\rho_{\omega,i,N}\int_{P_{i}(I^{n})}\underset{i}{\overset{I^{n}%
}{\mathbf{V}}}h_{i}dm+K_{\omega,i,N}\Vert h_{i}\Vert_{1}\text{.} \label{b1}%
\end{align}

The expression in $I_{2}$ is less than or equal to%
\begin{align*}
&  \underset{\overset{\rightarrow}{s}\in%
%TCIMACRO{\U{2124} }%
%BeginExpansion
\mathbb{Z}
%EndExpansion
_{\overset{\rightarrow}{r}}}{\sup}\delta_{\omega,i,N,\overset{\rightarrow}{s}}%
%TCIMACRO{\dsum \limits_{\overset{\rightarrow}{s}\in\mathbb{Z}_{_{\overset
%{\rightarrow}{r}}}}}%
%BeginExpansion
{\displaystyle\sum\limits_{\overset{\rightarrow}{s}\in\mathbb{Z}%
_{_{\overset{\rightarrow}{r}}}}}
%EndExpansion
\int_{P_{i}(I_{\omega,N,\overset{\rightarrow}{s}})}\Big(|h_{i}(\Psi
_{\omega,1,N,\overset{\rightarrow}{s}}(x_{1}),\dots,a_{i,s_{i}},\dots
,\Psi_{\omega,n,N,\overset{\rightarrow}{s}}(x_{n}))|\\
&  +|h_{i}(\Psi_{\omega,1,N,\overset{\rightarrow}{s}}(x_{1}),\dots
,a_{i,s_{i}-1},\dots,\Psi_{\omega,n,N,\overset{\rightarrow}{s}}(x_{n}))|\Big)%
%TCIMACRO{\dprod \limits_{\overset{j=1}{j\neq i}}^{n}}%
%BeginExpansion
{\displaystyle\prod\limits_{\overset{j=1}{j\neq i}}^{n}}
%EndExpansion
\delta_{\omega,j,N,\overset{\rightarrow}{s}}(x_{j})dm.
\end{align*}
Since $h\geq0$, the argument after Equation $(5)$ in \cite{P}, implies%
\begin{align*}
\Big(  &  |h_{i}(\Psi_{\omega,1,N,\overset{\rightarrow}{s}}(x_{1}%
),\dots,a_{i,s_{i}},\dots,\Psi_{\omega,n,N,\overset{\rightarrow}{s}}%
(x_{n}))|\\
&  +|h_{i}(\Psi_{\omega,1,N,\overset{\rightarrow}{s}}(x_{1}),\dots
,a_{i,s_{i}-1},\dots,\Psi_{\omega,n,N,\overset{\rightarrow}{s}}(x_{n}%
))|\Big)\\
\leq &  \underset{i}{\overset{P_{i}(I_{\omega,N,\overset{\rightarrow}{s}%
})\times I}{\mathbf{V}}}h_{i}(\Psi_{\omega,1,N,\overset{\rightarrow}{s}}%
(x_{1}),\dots,x_{i},\dots,\Psi_{\omega,n,N,\overset{\rightarrow}{s}}(x_{n}))\\
&  +2%
%TCIMACRO{\dint \limits_{0}^{1}}%
%BeginExpansion
{\displaystyle\int\limits_{0}^{1}}
%EndExpansion
h_{i}(\Psi_{\omega,1,N,\overset{\rightarrow}{s}}(x_{1}),\dots,x_{i},\dots
,\Psi_{\omega,n,N,\overset{\rightarrow}{s}}(x_{n}))dx_{i}\text{,}%
\end{align*}
then we have the expression in $I_{2}$ is less than or equal to
\begin{align*}
&  \rho_{\omega,i,N}%
%TCIMACRO{\dsum \limits_{\overset{\rightarrow}{s}\in\mathbb{Z}_{P_{i}%
%(\overset{\rightarrow}{r})}}}%
%BeginExpansion
{\displaystyle\sum\limits_{\overset{\rightarrow}{s}\in\mathbb{Z}%
_{P_{i}(\overset{\rightarrow}{r})}}}
%EndExpansion
\int_{P_{i}(I_{\omega,N,\overset{\rightarrow}{s}})}\Big(\underset{i}%
{\overset{P_{i}(I_{\omega,N,\overset{\rightarrow}{s}})\times I}{\mathbf{V}}%
}h_{i}(\Psi_{\omega,1,N,\overset{\rightarrow}{s}}(x_{1}),\dots,x_{i}%
,\dots,\Psi_{\omega,n,N,\overset{\rightarrow}{s}}(x_{n}))\\
&  +2%
%TCIMACRO{\dint \limits_{0}^{1}}%
%BeginExpansion
{\displaystyle\int\limits_{0}^{1}}
%EndExpansion
h_{i}(\Psi_{\omega,1,N,\overset{\rightarrow}{s}}(x_{1}),\dots,x_{i},\dots
,\Psi_{\omega,n,N,\overset{\rightarrow}{s}}(x_{n}))dx_{i}\Big)%
%TCIMACRO{\dprod \limits_{\overset{j=1}{j\neq i}}^{n}}%
%BeginExpansion
{\displaystyle\prod\limits_{\overset{j=1}{j\neq i}}^{n}}
%EndExpansion
\delta_{\omega,j,N,\overset{\rightarrow}{s}}(x_{j})dm.
\end{align*}

Using again the change of variables $U=\Psi^{-1}$, we get the last sum is
equal to%

\begin{align}
&  \rho_{\omega,i,N}%
%TCIMACRO{\dsum \limits_{\overset{\rightarrow}{s}\in\mathbb{Z}_{P_{i}%
%(\overset{\rightarrow}{r})}}}%
%BeginExpansion
{\displaystyle\sum\limits_{\overset{\rightarrow}{s}\in\mathbb{Z}%
_{P_{i}(\overset{\rightarrow}{r})}}}
%EndExpansion
\int_{P_{i}(B_{\overset{\rightarrow}{s}})}\Big(\underset{i}{\overset
{P_{i}(B_{\overset{\rightarrow}{s}})\times I}{\mathbf{V}}}h_{i}(x_{1}%
,\dots,x_{n})+2%
%TCIMACRO{\dint \limits_{0}^{1}}%
%BeginExpansion
{\displaystyle\int\limits_{0}^{1}}
%EndExpansion
h_{i}(x_{1},\dots,x_{n})dx_{i}\Big)dm\nonumber\\
&  =\rho_{\omega,i,N}\int_{P_{i}(I^{n})}\underset{i}{\overset{I^{n}%
}{\mathbf{V}}}h_{i}dm+2\rho_{\omega,i,N}\int_{P_{i}(I^{n})}%
%TCIMACRO{\dint \limits_{0}^{1}}%
%BeginExpansion
{\displaystyle\int\limits_{0}^{1}}
%EndExpansion
h_{i}dx_{i}dm\nonumber\\
&  \leq\rho_{\omega,i,N}\int_{P_{i}(I^{n})}\underset{i}{\overset{I^{n}%
}{\mathbf{V}}}h_{i}dm+2\rho_{\omega,i,N}\Vert h_{i}\Vert_{1}\text{.}
\label{b2}%
\end{align}
\qquad

Now, combining the results from \eqref{b1}$,$ \eqref{b2} and \eqref{bound}, we
get%
\[
\int_{P_{i}(I^{n})}\underset{i}{\overset{I^{n}}{\mathbf{V}}}\mathcal{L}%
_{\omega}^{(N)}h_{i}dm\leq3\rho_{\omega,i,N}\int_{P_{i}(I^{n})}\underset
{i}{\overset{I^{n}}{\mathbf{V}}}h_{i}dm+(K_{\omega,i,N}+2\rho_{\omega
,i,N})\Vert h_{i}\Vert_{1}\text{.}%
\]
Thus, letting
\begin{align}
\alpha_{1}(\omega)  &  =\underset{i=1,\dots,n}{\max}3\rho_{\omega,i,N}%
\text{,}\label{zezeee}\\
\alpha_{2}(\omega)  &  =\underset{i=1,\dots,n}{\max}(K_{\omega,i,N}%
+2\rho_{\omega,i,N})\text{,}\nonumber
\end{align}
we have, for each $i=1,\dots,n$,%
\[
\int_{P_{i}(I^{n})}\underset{i}{\overset{I^{n}}{\mathbf{V}}}\mathcal{L}%
_{\omega}^{(N)}hdm\leq\alpha_{1}(\omega)\int_{P_{i}(I^{n})}\underset
{i}{\overset{I^{n}}{\mathbf{V}}}hdm+\alpha_{2}(\omega)\Vert h\Vert_{1}\text{.}%
\]

By \cite[Lemma C.$5$]{CQ} and Lemma \ref{singmanergodic}, the index of
compactness $\mathcal{K}_{N}(\omega)$ is less than%
\[
\int_{\sigma^{-\ell}Z}\log\alpha_{1}(\bar{\omega})d\mathbb{P(\bar{\omega}%
)}\text{,}%
\]
where $\ell$ is such that $\sigma^{-\ell}Z$ is the ergodic component of
$\sigma^{N}$ containing $\omega$. Since $\Omega=%
%TCIMACRO{\dbigcup \limits_{s=0}^{k-1}}%
%BeginExpansion
{\displaystyle\bigcup\limits_{s=0}^{k-1}}
%EndExpansion
\sigma^{-\ell}Z$ and $\int_{\Omega}\log\alpha_{1}(\omega)d\mathbb{P(\omega
)<}0$, we have $\int_{\sigma^{-\ell_{0}}Z}\log\alpha_{1}({\omega
})d\mathbb{P({\omega})<}0$ for some $\ell_{0}=0,1,\dots,k-1$. By Proposition
\ref{kjn}, we have $\mathcal{K}^{\ast}=\frac{\mathcal{K}_{N}(\omega)}{N}<0$.

Since the transfer operator $\mathcal{L}_{\omega}^{(n)}$ is a Markov operator
for each $\omega\in\Omega$, for any density function $h\in BV(I^{n})$, we have
that $\Vert\mathcal{L}_{\omega}^{(n)}h\Vert_{BV}\geq\Vert\mathcal{L}_{\omega
}^{(n)}h\Vert_{1}=\Vert h\Vert_{1}=1$. This shows that
\begin{equation}
\lambda^{\ast}\geq0, \label{bsb}%
\end{equation}
and therefore $\mathcal{K}^{\ast}<\lambda^{\ast}$. This finishes the proof of
Theorem \ref{quco} (i).
\end{proof}

\begin{proof}
[Proof of Theorem \ref{quco} (ii)]In the proof of Theorem \ref{quco} (i), we
proved that there are $N\in%
%TCIMACRO{\U{2115} }%
%BeginExpansion
\mathbb{N}
%EndExpansion
$ where $N$ satisfies the condition in \eqref{Nc} and $\alpha_{1},\alpha
_{2}:\Omega\rightarrow%
%TCIMACRO{\U{211d} }%
%BeginExpansion
\mathbb{R}
%EndExpansion
^{+}$such that $\int_{\Omega}\log\alpha_{1}(\omega)d\mathbb{P(\omega)<}0$ with
the property that%
\begin{equation}
\overset{I^{n}}{\mathbb{V}}\mathcal{L}_{\omega}^{(N)}h\leq\alpha_{1}%
(\omega)\overset{I^{n}}{\mathbb{V}}h+\alpha_{2}(\omega)\Vert h\Vert
_{1}\text{,} \label{eq:LY}%
\end{equation}
for all $h\in BV(I^{n})$ and $\omega\in\Omega$. We also proved that
$\lambda^{\ast}\geq0$ in \eqref{bsb}. It remains to prove $\lambda^{\ast}%
\leq0$. Since $\Vert\mathcal{L}_{\omega}\Vert_{1}\leq1$, it is enough to
consider the growth of the variation of the term $\mathcal{L}_{\omega}^{(n)}%
h$. Using the argument in \cite[Lemma C.5]{CQ} and \cite[Proposition 1.4]{B},
$\alpha_{1}(\omega)$ and $\alpha_{2}(\omega)$ can be redefined so that
\eqref{eq:LY} holds and $\alpha_{2}(\omega)$ is uniformly bounded by positive
constant $\tilde{\alpha}_{2}$, which gives a hybrid Lasota-Yorke inequality%
\begin{equation}
\overset{I^{n}}{\mathbb{V}}\mathcal{L}_{\omega}^{(N)}h\leq\alpha_{1}%
(\omega)\overset{I^{n}}{\mathbb{V}}h+\tilde{\alpha}_{2}\Vert h\Vert
_{1}\text{.} \label{hlyi}%
\end{equation}
By iterating the hybrid Lasota-Yorke inequality \eqref{hlyi}, we get a bound
on the sequence $(\overset{I^{n}}{\mathbb{V}}\mathcal{L}_{\omega}%
^{(Nk)}h)_{k=1}^{\infty}$. Therefore,
\[
\lim_{k\rightarrow\infty}\frac{1}{Nk}\log\Vert\mathcal{L}_{\omega}%
^{(Nk)}h\Vert_{BV}\leq0\text{.}%
\]
and since this is true for almost every $\omega\in\Omega$,
Proposition~\ref{kjn} implies that $\lambda^{\ast}\leq0$.
\end{proof}

\section{Random invariant densities and ACIPs, skew product ACIPs and Physical
measures}

\label{S:4}

The concept of random invariant measures (for random dynamical systems) is a
natural generalization of the notion of invariant measures (for deterministic
dynamical systems). In this section we introduce our main results regarding
the existence of random invariant densities and measures as well as skew
product ACIPs. After that, we deduce the existence of physical measures. We
shall assume throughout the rest of the paper that $\int_{\Omega}\log^{+}%
\Vert\mathcal{L}_{\omega}\Vert_{BV}d\mathbb{P(\omega)<}\infty$.

\begin{definition}
Let $\mathscr{F}=\{f_{\omega}\}_{\omega\in\Omega}$ be an admissible random
Jab\l o\'{n}ski map. A family $\{\mu_{\omega}\}_{\omega\in\Omega}$ of random
invariant measures for $\mathscr{F}$ is a family of probability measures
$\mu_{\omega}$ on $I^{n}$ where the map $\omega\mapsto\mu_{\omega}$ is
measurable and%
\[
f_{\omega}\mu_{\omega}=\mu_{\sigma\omega}\text{, for }\mathbb{P}\text{-a.e.
}\omega\in\Omega\text{.}%
\]
A family $\{h_{\omega}\}_{\omega\in\Omega}$ of random invariant densities for
$\mathscr{F}$ is a family such that $h_{\omega}\geq0$, $h_{\omega}\in
L^{1}(I^{n})$, $\Vert h_{\omega}\Vert_{1}=1$, the map $\omega\mapsto
h_{\omega}$ is measurable and%
\begin{equation}
\mathcal{L}_{{\omega}}h_{\omega}=h_{\sigma\omega}\text{, for }\mathbb{P}%
\text{-a.e. }\omega\in\Omega\text{.} \label{eqivar}%
\end{equation}

\end{definition}

\begin{proposition}
\label{singmanergodic copy(3)} Let $N$ be as in \eqref{Nc}.
%Without assuming that the map $\sigma^{N}$ is
%ergodic, where $N$ as in the Lasota-Yorke inequality \eqref{eq:LY}.
Then, for $\mathbb{P}$-almost all $\omega\in\Omega$, we have
\[
\lim_{j\rightarrow\infty}\frac{1}{j}\sum_{t=1}^{j}\log(\alpha_{1}(\sigma
^{-tN}\omega))<0\text{.}%
\]

\end{proposition}

\begin{proof}
By Lemma \ref{singmanergodic}, there exists $k$, a factor of $N$, and a
$\sigma^{-N}$-invariant subset $Z$ of $\Omega$ of measure $1/k$ such that
$\Omega=%
%TCIMACRO{\dbigcup \limits_{s=0}^{k-1}}%
%BeginExpansion
{\displaystyle\bigcup\limits_{s=0}^{k-1}}
%EndExpansion
\sigma^{\ell}Z$ and $\sigma^{-N}|_{Z}$ is ergodic. In fact, since $\sigma$ is
invertible, ergodic and $\mathbb{P}$-preserving, $\sigma^{-N}|_{\sigma^{\ell
}Z}$ is ergodic and $\mathbb{P(\sigma}^{\ell}Z\mathbb{)=}\frac{1}{k}$, for all
$\ell=0,1,\dots,k-1$. By Birkhoff ergodic theorem, we have
\[
\lim_{j\rightarrow\infty}\frac{1}{j}\sum_{t=1}^{j}\log(\alpha_{1}(\sigma
^{-tN}\omega))=k\int_{\mathbb{\sigma}^{\ell}Z}\log\alpha_{1}(\bar
\omega)d\mathbb{P(\bar\omega)}\text{,}%
\]
for $\mathbb{P}$-almost all $\omega\in\sigma^{\ell}Z$, and $\ell
=0,1,\dots,k-1$. Note that for any $\ell=0,1,\dots,k-1$ and $\mathbb{P}%
$-almost all $\omega\in\sigma^{\ell}Z$, the definition of $\alpha_{1}(\omega)$
in \eqref{zezeee} and the argument in \eqref{lolololo} imply%
\begin{align*}
&  \int_{\sigma^{\ell}Z}\log\alpha_{1}(\omega)d\mathbb{P(\omega)=}\int
_{\sigma^{\ell}Z}\log\Big(\underset{i=1,\dots,n}{\max}3\Big(\underset
{\overset{\rightarrow}{s}\in%
%TCIMACRO{\U{2124} }%
%BeginExpansion
\mathbb{Z}
%EndExpansion
_{\overset{\rightarrow}{r}}}{\sup}|(\varphi_{\omega,i,N,\overset{\rightarrow
}{s}}^{-1})^{^{\prime}}|\Big)\Big)d\mathbb{P(\omega)}\\
&  =\int_{\sigma^{\ell}Z}\log\Big(\underset{i=1,\dots,n}{\max}3\Big(\underset
{\overset{\rightarrow}{s}\in%
%TCIMACRO{\U{2124} }%
%BeginExpansion
\mathbb{Z}
%EndExpansion
_{\overset{\rightarrow}{r}}}{\sup}|\Big((\varphi_{\sigma^{N-1}\omega
,i,\overset{\rightarrow}{s_{N-1}}}\circ\dots\circ\varphi_{\sigma
\omega,i,\overset{\rightarrow}{s_{1}}}\circ\varphi_{\omega,i,\overset
{\rightarrow}{s}})^{-1}\Big)^{^{\prime}}|\Big)\Big)d\mathbb{P(\omega)}\\
&  \leq\int_{\sigma^{\ell}Z}\log\Big(\underset{i=1,\dots,n}{\max}3\Big(|%
%TCIMACRO{\dprod \limits_{t=0}^{N-1}}%
%BeginExpansion
{\displaystyle\prod\limits_{t=0}^{N-1}}
%EndExpansion
\underset{\overset{\rightarrow}{s}\in%
%TCIMACRO{\U{2124} }%
%BeginExpansion
\mathbb{Z}
%EndExpansion
_{\overset{\rightarrow}{r}}}{\sup}\frac{1}{\varphi_{\sigma^{t}(\omega
),i,\overset{\rightarrow}{s_{t}}}^{\prime}(\varphi_{\omega,i,t-1,\overset
{\rightarrow}{s_{t-1}}}(x_{i}))}|\Big)\Big)d\mathbb{P(\omega)}\text{. }%
\end{align*}
By definition of $\mathbb{\gamma}_{i}$ in Equation \eqref{fec}, we have
\begin{align*}
\int_{\sigma^{\ell}Z}\log\alpha_{1}(\omega)d\mathbb{P(\omega)}  &  \leq
\int_{\sigma^{\ell}Z}\log3-\underset{i=1,\dots,n}{\min}\log\Big(%
%TCIMACRO{\dprod \limits_{t=0}^{N-1}}%
%BeginExpansion
{\displaystyle\prod\limits_{t=0}^{N-1}}
%EndExpansion
\mathbb{\gamma}_{i}(\sigma^{t}\omega)\Big)d\mathbb{P(\omega)}\\
&  =\frac{\log(3)}{k}\mathbb{-}%
%TCIMACRO{\dsum \limits_{t=0}^{N-1}}%
%BeginExpansion
{\displaystyle\sum\limits_{t=0}^{N-1}}
%EndExpansion
\int_{\sigma^{\ell}Z} \underset{i=1,\dots,n}{\min} \log(\mathbb{\gamma}%
_{i}(\sigma^{t}\omega))d\mathbb{P(\omega)}\text{.}%
\end{align*}
Since $\sigma$ is measure preserving, a change of variables makes the last
term equal to%
\begin{align*}
\frac{\log(3)}{k}  &  \mathbb{-}%
%TCIMACRO{\dsum \limits_{t=0}^{N-1}}%
%BeginExpansion
{\displaystyle\sum\limits_{t=0}^{N-1}}
%EndExpansion
\int_{Z} \underset{i=1,\dots,n}{\min} \log(\mathbb{\gamma}_{i}(\sigma^{t-\ell
}\omega))d\mathbb{P(\omega)}=\frac{\log(3)}{k}\mathbb{-}\frac{N}{k}%
\int_{\Omega}\underset{i=1,\dots,n}{\min} \log(\mathbb{\gamma}_{i}%
(\omega))d\mathbb{P(\omega)}\\
&  =\frac{1}{k}(\log(3)\mathbb{-}N\Gamma)<\frac{1}{k}(\log(3)\mathbb{-}%
N\gamma)<0\text{,}%
\end{align*}
by Definition \ref{arjm} and \eqref{Nc}.
\end{proof}

\begin{theorem}
\label{muer} Consider an admissible random Jab\l o\'{n}ski map $\mathscr F$.
For each $\omega\in\Omega\ $and $k=1,2,\dots$, we define
\[
h_{\omega}^{k}=(\mathcal{L}_{\sigma^{-1}\omega}\circ\dots\circ\mathcal{L}%
_{\sigma^{-(k-1)}\omega}\circ\mathcal{L}_{\sigma^{-k}\omega})1\text{,}%
\]
where $1\in BV(I^{n})$ is the constant function and for each $s=1,2,\dots$, we
define%
\[
H_{\omega}^{s}=\frac{1}{s}\sum_{k=1}^{s}h_{\omega}^{k}\text{.}%
\]
Then, for $\mathbb{P}$-a.e. $\omega\in\Omega$:\newline(i) the sequence
$\{H_{\omega}^{s}\}_{s\in\mathbb{N}}$ is relatively compact in $L^{1}$;
and\newline(ii) the following limit exists,%
\begin{equation}
\lim_{s\rightarrow\infty}H_{\omega}^{s}=:h_{\omega}\in BV(I^{n})\text{ in
}L^{1}\text{.} \label{hw}%
\end{equation}
Moreover, $\{h_{\omega}\}_{\omega\in\Omega}$ is a family of random invariant
densities for $\mathscr{F}$.
\end{theorem}

\begin{proof}
Recall from the proof of Theorem \ref{quco} (ii), there are $N\in%
%TCIMACRO{\U{2115} }%
%BeginExpansion
\mathbb{N}
%EndExpansion
$ where $N$ satisfies\ the condition in \eqref{Nc}, a constant $\tilde{\alpha
}_{2}$ and a positive measurable function $\alpha_{1}:\Omega\rightarrow%
%TCIMACRO{\U{211d} }%
%BeginExpansion
\mathbb{R}
%EndExpansion
^{+}$ such that $\int_{\Omega}\log\alpha_{1}(\omega)d\mathbb{P(\omega)<}0$ and
the hybrid Lasota-Yorke inequality \eqref{hlyi} is satisfied. That is,
\[
\overset{I^{n}}{\mathbb{V}}\mathcal{L}_{\omega}^{(N)}h\leq\alpha_{1}%
(\omega)\overset{I^{n}}{\mathbb{V}}h+\tilde{\alpha}_{2}\Vert h\Vert_{1},
\]
for all $h\in BV(I^{n})$ and $\omega\in\Omega$. For $k=1,2,\dots$, and
$\mathbb{P}$-almost all $\omega\in\Omega$,
%comparing $h^{Nk}_{\omega}$ to $\mathcal{L}%
%_{f_{\omega}^{(N)}}$ gives
the following holds,
\[
h_{\omega}^{Nk}=\mathcal{L}_{\sigma^{-Nk}\omega}^{(Nk)}1\text{.}%
\]
Applying \eqref{hlyi} to upper bound the variation of $h_{\omega}^{Nk}$ on
$I^{n}$ yields
\begin{align*}
\overset{I^{n}}{\mathbb{V}}h_{\omega}^{Nk}  &  =\overset{I^{n}}{\mathbb{V}%
}\mathcal{L}_{\sigma^{-Nk}\omega}^{(Nk)}1\\
&  \leq\alpha_{1}(\sigma^{-N}\omega)\overset{I^{n}}{\mathbb{V}}(\mathcal{L}%
_{\sigma^{-(N+1)}(\omega)}\circ\dots\circ\mathcal{L}_{\sigma^{-2N}(\omega
)}\circ\mathcal{L}_{\sigma^{-Nk}(\omega)})1\\
&  \ +\tilde{\alpha}_{2}\Vert(\mathcal{L}_{\sigma^{-(N+1)}(\omega)}\circ
\dots\circ\mathcal{L}_{\sigma^{-2N}(\omega)}\circ\mathcal{L}_{\sigma
^{-Nk}(\omega)})1\Vert_{1}\\
&  \leq\alpha_{1}(\sigma^{-N}\omega)\alpha_{1}(\sigma^{-2N}\omega
)\overset{I^{n}}{\mathbb{V}}(\mathcal{L}_{\sigma^{-(2N+1)}(\omega)}\circ
\dots\circ\mathcal{L}_{\sigma^{-3N}(\omega)}\circ\mathcal{L}_{\sigma
^{-Nk}(\omega)})1\\
&  \ +\tilde{\alpha}_{2}\Vert(\mathcal{L}_{\sigma^{-(N+1)}(\omega)}\circ
\dots\circ\mathcal{L}_{\sigma^{-2N}(\omega)}\circ\mathcal{L}_{\sigma
^{-Nk}(\omega)})1\Vert_{1}\\
&  \ +\alpha_{1}(\sigma^{-N}\omega)\tilde{\alpha}_{2}\Vert(\mathcal{L}%
_{\sigma^{-(2N+1)}(\omega)}\circ\dots\circ\mathcal{L}_{\sigma^{-3N}(\omega
)}\circ\mathcal{L}_{\sigma^{-Nk}(\omega)})1\Vert_{1}\\
&  \leq\dots\leq\alpha_{1}(\sigma^{-N}\omega)\alpha_{1}(\sigma^{-2N}%
\omega)\dots\alpha_{1}(\sigma^{-kN}\omega)\overset{I^{n}}{\mathbb{V}}1\\
&  \ +\tilde{\alpha}_{2}\Vert(\mathcal{L}_{\sigma^{-(N+1)}(\omega)}\circ
\dots\circ\mathcal{L}_{\sigma^{-2N}(\omega)}\circ\mathcal{L}_{\sigma
^{-Nk}(\omega)})1\Vert_{1}\\
&  \ +\alpha_{1}(\sigma^{-N}\omega)\tilde{\alpha}_{2}\Vert(\mathcal{L}%
_{\sigma^{-(2N+1)}(\omega)}\circ\dots\circ\mathcal{L}_{\sigma^{-3N}(\omega
)}\circ\mathcal{L}_{\sigma^{-Nk}(\omega)})1\Vert_{1}\\
&  \ +\dots+\alpha_{1}(\sigma^{-N}\omega)\alpha_{1}(\sigma^{-2N}\omega
)\dots\alpha_{1}(\sigma^{-kN}\omega)\tilde{\alpha}_{2}\Vert1\Vert_{1}\text{,}%
\end{align*}
and since $\overset{I^{n}}{\mathbb{V}}1=0$, $\Vert1\Vert_{1}=1$ and the
transfer operator is contractive, we have%

\begin{align*}
\overset{I^{n}}{\mathbb{V}}h_{\omega}^{Nk}  &  \leq\tilde{\alpha}%
_{2}\Big(1+\alpha_{1}(\sigma^{-N}\omega)+\alpha_{1}(\sigma^{-N}\omega
)\alpha_{1}(\sigma^{-2N}\omega)+\dots\\
&  \ +\alpha_{1}(\sigma^{-N}\omega)\alpha_{1}(\sigma^{-2N}\omega)\dots
\alpha_{1}(\sigma^{-kN}\omega)\Big)\\
&  =\tilde{\alpha}_{2}(1+\sum_{j=1}^{k}\alpha_{1}^{(j)}(\sigma^{-jN}%
\omega))\text{,}%
\end{align*}
where for $j=1,2,\dots$, we let $\alpha_{1}^{(j)}(\sigma^{-jN}\omega
)=\alpha_{1}(\sigma^{-N}\omega)\alpha_{1}(\sigma^{-2N}\omega)\dots\alpha
_{1}(\sigma^{-jN}\omega)$.
%Note that for $j=1,2,\dots$, we have%
%\[
%\frac{1}{j}\log\alpha_{1}^{(j)}(\sigma^{-jN}\omega)=\frac{1}{j}\log\Big(\alpha_{1}%
%(\sigma^{-N}\omega)\alpha_{1}(\sigma^{-2N}\omega)\dots\alpha_{1}(\sigma
%^{-jN}\omega)\Big)=\frac{1}{j}\sum_{t=1}^{j}\log\alpha_{1}(\sigma^{-tN}%
%\omega)\text{.}%
%\]
By Proposition \ref{singmanergodic copy(3)}, there exists $0<\hat{\alpha
}(\omega)<1$ such that the time averages $\frac{1}{j}\log\alpha_{1}%
^{(j)}(\sigma^{-jN}\omega)$ converge to $\log(\hat{\alpha}(\omega))<0$. Choose
$\alpha(\omega)$ such that $0<\hat{\alpha}(\omega)<\alpha(\omega)<1$. For
sufficiently large $j_{0}(\omega)$, we have that%
\[
\alpha_{1}^{(j)}(\sigma^{-jN}\omega)<\alpha(\omega)^{j}\text{, for all }j\geq
j_{0}(\omega)\text{.}%
\]
Let $c(\omega)$ be defined as%
\[
c(\omega)=\underset{1\leq j\leq j_{0}(\omega)}{\max}(\frac{\alpha_{1}%
^{(j)}(\sigma^{-jN}\omega)}{\alpha(\omega)^{j}},1)\text{,}%
\]
and hence for all $j$, we have that%
\[
\alpha_{1}^{(j)}(\sigma^{-jN}\omega)<c(\omega)\alpha(\omega)^{j}\text{.}%
\]
Taking the sum over $j$, we get that%
\begin{align*}
\tilde{\alpha}_{2}(1+\sum_{j=1}^{k}\alpha_{1}^{(j)}(\sigma^{-jN}\omega))  &
\leq\tilde{\alpha}_{2}(1+c(\omega)\sum_{j=0}^{\infty}\alpha(\omega)^{j})\\
&  =\tilde{\alpha}_{2}(1+c(\omega)\tilde{\alpha}(\omega))\text{,}%
\end{align*}
where $\tilde{\alpha}(\omega)=\frac{1}{1-\alpha(\omega)}$. Let
\[
c_{1}(\omega)=\tilde{\alpha}_{2}(1+c(\omega)\tilde{\alpha}(\omega))\text{,}%
\]
then we have proven that for every $k\in\mathbb{N}$
\[
\overset{I^{n}}{\mathbb{V}}h_{\omega}^{Nk}\leq c_{1}(\omega)\text{.}%
\]
From this inequality, it follows $\{\overset{I^{n}}{\mathbb{V}}h_{\omega}%
^{Nk}\}_{k\in\mathbb{N}}$ is bounded. The same holds for the whole sequence
$\{\overset{I^{n}}{\mathbb{V}}h_{\omega}^{k}\}_{k\in\mathbb{N}}$, and indeed
for the averages $\{\overset{I^{n}}{\mathbb{V}}H_{\omega}^{s}\}_{s\in
\mathbb{N}}$. Hence, $\{H_{\omega}^{s}\}_{s\in\mathbb{N}}$ is relatively
compact in $L^{1}$ by \cite[Lemma A.$1$]{liv}. This establishes (i).

Then, the random mean ergodic theorem \cite[Theorem B]{NakamuraToyokawa} shows
that $\{H_{\omega}^{s}\}_{s\in\mathbb{N}}$ converges in the strong sense to a
random invariant density $h_{\omega}$, as in \eqref{hw}. The fact that
$h_{\omega}\in BV(I^{n})$ follows once again from the relative compactness of
$BV(I^{n})$ in $L^{1}$. This establishes (ii).
\end{proof}

We can think of the above random invariant densities $h_{\omega}$ as
asymptotic distributions arrived at by running the dynamics of a uniform
distribution from the distant past. Returning to the present setting of random
compositions of Jab\l o\'{n}ski maps, a family of random invariant measures
with densities of bounded variation will also define a measure that is
invariant with respect to the associated skew product, as described in the
following remark.

\begin{remark}
\label{re5} For $\mathbb{P}$-a.e. $\omega\in\Omega$, define $\mu_{\omega}$ on
the fiber $\{\omega\}\times I^{n}\subset\Omega\times I^{n}$, as
\[
\frac{d\mu_{\omega}}{dm}=h_{\omega}\text{,}%
\]
where $h_{\omega}$ is given by \eqref{hw}. Then $\mu_{\omega}$ is a random
invariant ACIP and the measure $\mu$ defined on $\mathbb{P\times}m$-measurable
sets $A\subseteq$ $\Omega\times I^{n}$ by
\[
\mu(A)=%
%TCIMACRO{\dint \limits_{\Omega}}%
%BeginExpansion
{\displaystyle\int\limits_{\Omega}}
%EndExpansion
\mu_{\omega}(A)d\mathbb{P(\omega)}\text{,}%
\]
is an ACIP for the associated skew product $F$ defined in \eqref{sp}.
\end{remark}

Multiplicative ergodic theorems are concerned with random dynamical systems
$\mathcal{R}= (\Omega,\mathcal{F},\mathbb{P},\sigma,\mathcal{X},\mathcal{L}%
\mathbb{)}$.
%$\sigma:\Omega
%\circlearrowleft$ where for each $\omega\in\Omega$ there is an operator
%$\mathcal{L}_{\omega}$ acting on a Banach space $\mathcal{X}$.
%One then
%studies properties of the operators $\mathcal{L}_{\omega}^{(n)}$,
They give rise to an $\omega$-dependent hierarchical decomposition of
$\mathcal{X}$ into equivariant subspaces, called Oseledets spaces. In the
literature, multiplicative ergodic theorems are divided into two types,
according to the invertibility of the base map $\sigma$ and the operators
$\mathcal{L}_{\omega}$. In \cite{FLQ}, Froyland, Lloyd and Quas show a
semi-invertible multiplicative ergodic theorem, where the base is assumed to
be invertible, but there is no assumption about invertibility of the operators
$\mathcal{L}_{\omega}$. We will apply this theorem to show that the random
invariant densities $h_{\omega}$ found in Theorem~\ref{muer} belong to the
leading Oseledets subspace. Moreover, we will deduce the finiteness of the
number of ergodic ACIPs in Corollary \ref{coro1}.

An Oseledets splitting for a random dynamical system $\mathcal{R}%
=(\Omega,\mathcal{F},\mathbb{P},\sigma,\mathcal{X},\mathcal{L}\mathbb{)}$
consists of

\begin{itemize}
\item A sequence of isolated (exceptional) Lyapunov exponents
\[
\infty>\lambda^{*}=\lambda_{1}>\lambda_{2}>\dots>\lambda_{l}>\mathcal{K}%
^{\ast}\geq-\infty\text{,}%
\]
where the index $l$ $\geq1$ is allowed to be finite or countably infinite, and

\item A family of $\omega$-dependent splittings,%
\begin{equation}
\mathcal{X}=Y_{1}(\omega)\oplus\dots\oplus Y_{l}(\omega)\oplus V(\omega
)\text{,} \label{y1}%
\end{equation}
where for $j=1,\dots,l$, $d_{j}:=\dim(Y_{j}(\omega))<\infty$ and $V(\omega
)\in\mathcal{G(X)}$ where $\mathcal{G(X)}$ is the Grassmannian of
$\mathcal{X}$.
\end{itemize}

%The $Y_{j}(\omega)$'s and $V(\omega)$ depend measurably on $\omega$,
For all $j=1,\dots,l$ and $\mathbb{P}$-a.e. $\omega\in\Omega$, we have%
\begin{align}
\mathcal{L}_{\omega}Y_{j}(\omega)  &  =Y_{j}(\sigma\omega)\text{,}\\
\mathcal{L}_{\omega}V(\omega)  &  \subseteq V(\sigma\omega)\text{,}%
\end{align}
and
\begin{align}
\lim_{s\rightarrow\infty}\frac{1}{s}\log\Vert\mathcal{L}_{\omega}^{(s)}y\Vert
&  =\lambda_{j}\text{, }\forall y\in Y_{j}(\omega)\backslash\{0\}\text{,}%
\label{keke}\\
\lim_{s\rightarrow\infty}\frac{1}{s}\log\Vert\mathcal{L}_{\omega}^{(s)}v\Vert
&  \leq\mathcal{K}^{\ast}\text{, }\forall v\in V(\omega)\text{.}%
\end{align}

\begin{theorem}
[{Froyland, Lloyd and Quas \cite[Theorem $17$]{FLQ}}]\label{thm:FLQ}%
\label{flq} Let $\Omega$ be a Borel subset of a separable complete metric
space, $\mathcal{F}$ the Borel sigma-algebra and $\mathbb{P}$ a Borel
probability measure. Let $\mathcal{X}$ be a Banach space. Consider a random
dynamical system $\mathcal{R}$ $=$ $(\Omega,\mathcal{F},\mathbb{P}%
,\sigma,\mathcal{X},\mathcal{L})$ with base transformation $\sigma
:\Omega\circlearrowleft$ an ergodic homeomorphism, and suppose that the
generator $\mathcal{L}:\Omega$ $\rightarrow L(\mathcal{X},\mathcal{X})$ is
$\mathbb{P}$-continuous and satisfies
\[
\int_{\Omega}\log^{+}\Vert\mathcal{L}_{\omega}\Vert d\mathbb{P(\omega)<}%
\infty\text{.}%
\]
If $\mathcal{R}$ is quasi-compact, that is, if $\mathcal{K}^{\ast}%
<\lambda^{\ast}$, then $\mathcal{R}$ admits a unique $\mathbb{P}$-continuous
Oseledets splitting.
\end{theorem}

%In the next corollary we show how the Oseledets splitting in the last theorem
%gives information on the invariant measures of random dynamical systems.

By Theorem~\ref{quco}, admissible random Jab\l o\'{n}ski maps give rise to
quasi-compact random dynamical systems with $\lambda_{1}=0$. Therefore,
Theorem~\ref{thm:FLQ} implies the following.

\begin{corollary}
\label{coro1} For $\mathbb{P}$-a.e. $\omega\in\Omega$, the random invariant
density $h_{\omega}$ given in \eqref{hw} belongs to the Oseledets space
$Y_{1}(\omega)$ given in \eqref{y1}. Moreover, the number $r$ of ergodic ACIPs
$\mu_{1},\dots,\mu_{r}$ with respect to the associated skew product is finite;
indeed, we have
\begin{equation}
r\leq d_{1}=\dim(Y_{1}(\omega))\text{.} \label{easybd}%
\end{equation}

\begin{proof}
Let $\omega\in\Omega$, by the equivariance property given in \eqref{eqivar},
we have $\mathcal{L}_{\omega}^{(m)}h_{\omega}=h_{\sigma^{m}\omega}$, for $m\in%
%TCIMACRO{\U{2115} }%
%BeginExpansion
\mathbb{N}
%EndExpansion
$. To show that $h_{\omega}\in Y_{1}(\omega)$, we verify the limit condition
given in \eqref{keke} for $j=1$. Note that%

\begin{align*}
&  \lim_{m\rightarrow\infty}\frac{1}{m}\log\Vert\mathcal{L}_{\omega}%
^{(m)}h_{\omega}\Vert_{BV}\\
&  =\lim_{m\rightarrow\infty}\frac{1}{m}\log\Vert h_{\sigma^{m}\omega}%
\Vert_{BV}\\
&  \geq\lim_{m\rightarrow\infty}\frac{1}{m}\log\Vert h_{\sigma^{m}\omega}%
\Vert_{1}=0=\lambda^{\ast}\text{,}%
\end{align*}
on the other hand%
\begin{align*}
&  \lim_{m\rightarrow\infty}\frac{1}{m}\log\Vert\mathcal{L}_{\omega}%
^{(m)}h_{\omega}\Vert_{BV}\\
&  \leq\lim_{m\rightarrow\infty}\frac{1}{m}\log\Vert\mathcal{L}_{\omega}%
^{(m)}\Vert_{BV}=0=\lambda^{\ast}\text{,}%
\end{align*}
by Theorem \ref{quco}. Since the splitting in Theorem \ref{flq} is unique,
this gives that $h_{\omega}\in Y_{1}(\omega)$. By the finite dimensionality of
the leading Oseledets subspace $Y_{1}(\omega)$, we get the bound given in \eqref{easybd}.
\end{proof}
\end{corollary}

%Invariant measures should be of physical interest. For instance, a point
%measure which is supported on a specific point of a given dynamical system is
%an invariant measure but such a measure tells nothing about the dynamics of
%the system around or away from this point.
Next, we define physical measures and show how the measures given in Corollary
\ref{coro1} are physical measures.

\begin{definition}
\label{phme} Consider the tuple $(\Omega,\mathcal{F},\mathbb{P},\sigma,f)$
where $(\Omega,\mathcal{F},\mathbb{P)}$ is a probability space, $\sigma
:\Omega\circlearrowleft$ an invertible, ergodic and $\mathbb{P-}$preserving
transformation and $f=\{f_{\omega}:M\rightarrow M\}_{\omega\in\Omega}$ where
$M\subseteq%
%TCIMACRO{\U{211d} }%
%BeginExpansion
\mathbb{R}
%EndExpansion
^{n}$. A probability measure $\nu$ on $M$ is called \emph{physical} if for
$\mathbb{P}$-a.e. $\omega\in\Omega$, the Lebesgue measure of the random basin
$RB_{\omega}(\nu)$ of $\nu$ at $\omega$ is positive where%
\[
RB_{\omega}(\nu)=\{x\in M:\frac{1}{s}\sum_{k=0}^{s-1}\delta_{f_{\omega}%
^{(k)}(x)}\rightarrow\nu\}\text{,}%
\]
where $\delta_{x}$ is the Dirac measure at a point $x$.
\end{definition}

The convergence in Definition \ref{phme} is in the weak convergence sense. In
the case where $f_{\omega}$ is independent of $\omega$, this reduces to the
definition of physical measure for a deterministic dynamical system. The next
probabilistic result due to Buzzi applies in our setting.

\begin{theorem}
[{Buzzi \cite[Proposition $4.1$]{B}}]\label{ph} Let $\mu_{i}$ be one of the
measures $\mu_{i}:$ $i=1,\dots r$ given in Corollary \ref{coro1}. Then, the
marginal measure of $\mu_{i}$ on $I^{n}$, denoted by $\nu_{i}$, is a physical
measure on $I^{n}$.
\end{theorem}

The union of all basins of the of the physical measures $\nu_{i}$ coming from
the marginals of $\mu_{i}$ on $I^{n}$, $i=1,\dots r$\ has full Lebesgue
measure, which means Lebesgue almost everywhere, the asymptotic long term
behaviour of the random orbits will be described by one of these physical
measures. Another immediate consequence of the proof of Theorem~\ref{ph} is
the following.

\begin{corollary}
There exists a constant $b>0$ such that for $\mathbb{P}$-a.e. $\omega\in
\Omega$ and $i=1,\dots r$, $m(RB_{\omega}(\nu_{i}))>b$.
\end{corollary}

%\begin{proof}
%Since $r$ is finite by Corollary \ref{coro1} and $m(%
%%TCIMACRO{\dbigcup \limits_{i=1}^{r}}%
%%BeginExpansion
%{\displaystyle\bigcup\limits_{i=1}^{r}}
%%EndExpansion
%RB_{\omega}(\nu_{i}))=m(I^{n})=1$ is also finite by the proof of Theorem
%\ref{ph}, such $b$ exists.
%\end{proof}

\section{Bounds on the number of ergodic skew product ACIPs\label{S:bounds}}

A difficulty in the general study of ACIPs of piecewise expanding maps in
higher dimensions is that the geometric complexity around discontinuities or
interior crossing points might grow rapidly as the dynamical partitions are
refined \cite{noorinf}. This is in contrast to one-dimensional maps, where the
geometry is much simpler and such a complexity growth can not happen. However,
this complication does not occur in the context of random Jab\l o\'{n}ski
maps. In \cite{GBP}, G\'{o}ra, Boyarsky and Proppe proved that for a class of
deterministic dynamical systems modeled by Jab\l o\'{n}ski transformations,
the number of crossing points gives an upper bound for the number of ergodic ACIPs.

In this section, we establish bounds on the number of ergodic ACIPs for random
Jab\l o\'{n}ski maps. The first bound, presented in Section~\ref{S:1stbound},
is motivated by the work of Buzzi \cite{B} in the one dimensional case of
random Lasota-Yorke maps. The second bound, presented in
Section~\ref{S:2ndbound} is inspired by the work of G\'{o}ra, Boyarsky and
Proppe on absolutely continuous invariant measures for deterministic dynamical
systems given by multidimensional expanding maps \cite{GBP}. An example is
presented in Section~\ref{S:ex}.

Let $\mathscr{F}=\{f_{\omega}\}_{\omega\in\Omega}$ be an admissible random
Jab\l o\'{n}ski map. Suppose that there exist $r$ mutually singular ergodic
ACIPs $\mu_{1},\dots,\mu_{r}$ for the associated skew product map $F$. Fix
$i\in\{1,\dots,r\}$ and $\omega\in\Omega$, then the fiber measure $\mu
_{\omega}^{i}$ is a measure on $I^{n}$. By Theorem $2$ in \cite{GBP}, the
support $\text{Supp}(\mu_{\omega}^{i})$\ of $\mu_{\omega}^{i}$ is open
Lebesgue almost everywhere. This fact was before introduced in Keller's thesis
\cite{Ke}. Let $I_{i,\omega}(0)\subseteq\text{Supp}(\mu_{\omega}^{i})$ be a
nontrivial rectangle lying inside one of the rectangles of $\mathcal{B}%
^{\omega}$. Define the sequence%
\begin{equation}
I_{i,\omega}(s+1)=f_{\sigma^{s}\omega}(I_{i,\omega}(s))\cap J\text{, }s\in%
%TCIMACRO{\U{2115} }%
%BeginExpansion
\mathbb{N}
%EndExpansion
\cup\{0\}\text{,} \label{Push}%
\end{equation}
where $J$ is the open rectangle in the partition $\mathcal{B}^{\sigma
^{s+1}\omega}$ of the Jab\l o\'{n}ski map $f_{\sigma^{s+1}\omega}$ which
maximizes the Lebesgue measure of $I_{i,\omega}(s+1)$. For $s\in%
%TCIMACRO{\U{2115} }%
%BeginExpansion
\mathbb{N}
%EndExpansion
\cup\{0\}$, define $c_{i,\omega}(s)$ to be the number of crossing points in
the partition $\mathcal{B}^{\sigma^{s+1}\omega}$ lying inside the image
$f_{\sigma^{s}\omega}(I_{i,\omega}(s))$. Let
\begin{equation}
M(\omega)=\max_{z\in%
%TCIMACRO{\U{211d} }%
%BeginExpansion
\mathbb{R}
%EndExpansion
}\max_{d=1,\dots,n}\{\text{number of rectangles }B\in\mathcal{B}^{\sigma
\omega}\text{ s.t. }H_{n-1}^{(d)}(z)\cap\,\text{Int}(B)\neq\phi\}\text{,}
\label{eq:Mom}%
\end{equation}
where $H_{n-1}^{(d)}(z)$ is the $(n-1)$ dimensional hyperplane given by the
equation $x_{d}=z$. This definition of $M$ is motivated by a deterministic
analogue, Definition $3$ in \cite{GBP}.

For $i=1,\dots,r$, denote by
\[
D_{i}=\{\omega\in\Omega:\text{Supp}(\mu_{\omega}^{i})\text{ has a crossing
point in its interior}\}\text{.}%
\]
Also, let
\begin{equation}
\mathbb{\gamma}(\omega)=%
%TCIMACRO{\dprod \limits_{i=1}^{n}}%
%BeginExpansion
{\displaystyle\prod\limits_{i=1}^{n}}
%EndExpansion
\mathbb{\gamma}_{i}(\omega)\text{,} \label{shedo}%
\end{equation}
and $\mathbb{\gamma}_{i}(\omega)$ is defined in equation \eqref{fec}.

\subsection{Multidimensional bound \`a la Buzzi}

\label{S:1stbound} In this section, we assume the following.
%there exists a constant $\delta$ such that%
\begin{equation}
\delta:=%
%TCIMACRO{\dint \limits_{\Omega}}%
%BeginExpansion
{\displaystyle\int\limits_{\Omega}}
%EndExpansion
\log(\frac{\mathbb{\gamma}(\omega)}{M(\omega)})d\mathbb{P(\omega)>}0\text{,}
\label{edp}%
\end{equation}
This condition means that, on average, the fiber expansion constants dominate
the partition complexities.

\begin{lemma}
\label{L51} Let $\mathscr{F}=\{f_{\omega}\}_{\omega\in\Omega}$ be an
admissible random Jab\l o\'{n}ski map and assume that \eqref{edp} is
satisfied. Then, the number $r$ of mutually singular ergodic ACIPs for the
associated skew product map $F$ satisfies
\begin{equation}
\int_{\Omega}\log\big(2^{n-1}(\frac{c_{t}(\omega)}{r}+1)\big)d\mathbb{P(\omega
)\geq\delta}\text{.} \label{maryam}%
\end{equation}

\end{lemma}

\begin{proof}
First we show that at least one of the sets in
\begin{equation}
f_{\sigma^{s}\omega}(I_{i,\omega}(s))\text{,} \quad s\in\mathbb{N }\cup\{0\}
\label{sets}%
\end{equation}
has a crossing point in its interior. The argument proceeds by contradiction.
Suppose that for none of the sets in \eqref{sets}
%$s\in%
%%TCIMACRO{\U{2115} }%
%%BeginExpansion
%\mathbb{N}
%%EndExpansion
%\cup\{0\}$, the set%
has a crossing point in the interior. Then,
\begin{align}
m(I_{i,\omega}(s+1))  &  \geq\frac{\mathbb{\gamma}(\sigma^{s}\omega)}%
{M(\sigma^{s}\omega)}m(I_{i,\omega}(s))\nonumber\\
&  \geq\frac{\mathbb{\gamma}(\sigma^{s}\omega)}{M(\sigma^{s}\omega)}%
\text{.}\dots\text{.}\frac{\mathbb{\gamma}(\omega)}{M(\omega)}m(I_{i,\omega
}(0))\text{.}\nonumber
\end{align}
By \eqref{edp}, we have $\delta=\int_{\Omega}\log(\frac{\mathbb{\gamma}%
(\omega)}{M(\omega)})d\mathbb{P(\omega)>}0$. Hence, Birkhoff ergodic theorem
implies that $m(I_{i,\omega}(s+1))\rightarrow\infty$ as $s\rightarrow\infty$,
and this\ is a contradiction. Hence, at least one of the sets in \eqref{sets}
has a crossing point in its interior.

For $k=0,1,2,3,\dots$ and $\omega\in\Omega$, define%
\[
g_{i,k}(\omega)=\left\{
\begin{array}
[c]{cc}%
\begin{array}
[c]{c}%
\frac{\mathbb{\gamma}(\sigma^{k}\omega)}{2^{n-1}(c_{i,\omega}(k)+1)}\\
\\
\end{array}
&
\begin{array}
[c]{c}%
:\sigma^{k}\omega\in D_{i}\\
\\
\end{array}
\\
\frac{\mathbb{\gamma}(\sigma^{k}\omega)}{M(\sigma^{k}\omega)} & :\sigma
^{k}\omega\in\Omega\backslash D_{i}%
\end{array}
\right.  \text{.}%
\]
By equation \eqref{Push}, for $s\in%
%TCIMACRO{\U{2115} }%
%BeginExpansion
\mathbb{N}
%EndExpansion
$, $I_{i,\omega}(s)$ comes from evolving $I_{i,\omega}(s-1)$ by the map
$f_{\sigma^{s-1}\omega}$ and then taking the largest intersection of its image
with one of the partition rectangles of $\mathcal{B}^{\sigma^{s}\omega}$.
Therefore, the volume of $I_{i,\omega}(s)$ depends on whether the set
$f_{\sigma^{s-1}\omega}(I_{i,\omega}(s-1))$ has a crossing point in its
interior or not. In case the interior of this set has a crossing point, the
volume of $I_{i,\omega}(s)$ is bounded below by the volume of $I_{i,\omega
}(s-1)$ expanded by $\mathbb{\gamma}(\sigma^{s-1}\omega)$\ and scaled by
$2^{n-1}(c_{i,\omega}(s-1)+1)$. This last scaling term is an upper bound on
the number of rectangles of $\mathcal{B}^{\sigma^{s}\omega}$ meeting
$f_{\sigma^{s-1}\omega}(I_{i,\omega}(s-1))$. On the other hand, if the
interior of $f_{\sigma^{s-1}\omega}(I_{i,\omega}(s-1))$ has no crossing
points, the volume of $I_{i,\omega}(s)$ is bounded below by the volume of
$I_{i,\omega}(s-1)$ expanded by $\mathbb{\gamma}(\sigma^{s-1}\omega)$\ and
scaled by $M(\sigma^{s-1}\omega)$. Thus, in general,%
\[
m(I_{i,\omega}(s))\geq g_{i,s-1}(\omega)m(I_{i,\omega}(s-1))\text{.}%
\]
Therefore, inductively, we have%
\begin{equation}
m(I_{i,\omega}(s))\geq g_{i,s-1}(\omega)\text{.}\dots\text{.}g_{i,0}%
(\omega)m(I_{i,\omega}(0))\text{.} \label{ala}%
\end{equation}
Since $m(I_{i,\omega}(s))\leq1$, for all $s=1,2,3,\dots$, we have%
\begin{equation}
\sum_{k=0}^{s-1}\log(\frac{1}{g_{i,k}(\omega)})\geq\log(m(I_{i,\omega
}(0))\text{.} \label{e1}%
\end{equation}
By summing over $i=1,\dots,r$ and dividing by $r$, we get%
\[
\sum_{k=0}^{s-1}\frac{1}{r}\sum_{i=1}^{r}\log(\frac{1}{g_{i,k}(\omega)}%
)\geq\xi\text{,}%
\]
where $\xi:=\frac{1}{r}\sum_{i=1}^{r}\log(m(I_{i,\omega}(0))$. This gives
that
\begin{equation}
\sum_{k=0}^{s-1}\log(\frac{1}{(g_{1,k}(\omega)\text{.}\dots\text{.}%
g_{r,k}(\omega))^{\frac{1}{r}}})\geq\xi\text{.} \label{beforedividing}%
\end{equation}
Since the measures $\mu_{i}$ are mutually singular, for all $\omega\in\Omega$,
we have%

\[
c_{1,\omega}(k)+\dots+c_{r,\omega}(k)\leq c_{t}(\sigma^{k}\omega)\text{,}%
\]
where we recall that $c_{t}(\omega)$ is the total number of interior crossing
points in the partition $\mathcal{B}^{\sigma\omega}$ of $f_{\sigma\omega}$. By
adding $r$ to both sides, dividing by $r$ and using the arithmetic-geometric
mean inequality, we get%
\[
\Big((c_{1,\omega}(k)+1)\text{.}\dots\text{.}(c_{r,\omega}(k)+1)\Big)^{\frac
{1}{r}}\leq\frac{c_{t}(\sigma^{k}\omega)+r}{r}\text{.}%
\]
Therefore, \eqref{beforedividing} and the definition of $g_{i,k}(\omega)$
yield%
\begin{align}
&  \frac{1}{s}\sum_{k=0}^{s-1}\log\Big(\frac{2^{n-1}(\frac{c_{t}(\sigma
^{k}\omega)+r}{r})M(\sigma^{k}\omega)}{\mathbb{\gamma}(\sigma^{k}\omega
)}\Big)\nonumber\\
&  \geq\frac{1}{s}\sum_{k=0}^{s-1}\log\Big(\frac{1}{(g_{1,k}(\omega
)\text{.}\dots\text{.}g_{r,k}(\omega))^{\frac{1}{r}}}\Big)\geq\frac{\xi}%
{s}\text{.}%
\end{align}
Applying Birkhoff ergodic theorem,
%and using \eqref{beforedividing}
we get%
\[
\int_{\Omega}\log\Big(\frac{2^{n-1}(\frac{c_{t}(\omega)}{r}+1)M(\omega
)}{\mathbb{\gamma}(\omega)}\Big)d\mathbb{P(\omega)\geq}0\text{.}%
\]
This gives that%
\[
\int_{\Omega}\log(2^{n-1}(\frac{c_{t}(\omega)}{r}+1))d\mathbb{P(\omega)}%
+\int_{\Omega}\log(\frac{M(\omega)}{\mathbb{\gamma}(\omega)})d\mathbb{P}%
(\omega)\mathbb{\geq}0\text{,}%
\]
and therefore we have
\begin{equation}
\int_{\Omega}\log\big(2^{n-1}(\frac{c_{t}(\omega)}{r}+1)\big)d\mathbb{P(\omega
)\geq\delta}\text{.}%
\end{equation}

\end{proof}

Lemma~\ref{L51} may be used to obtain explicit bounds on $r$.

\begin{lemma}
\label{lem:bound1} \label{bbbb}Suppose \eqref{edp} holds, for $\mathbb{P}%
$-a.e. $\omega\in\Omega$, $c_{t}(\omega)\leq c$ and $\log(2^{n-1}%
)\mathbb{<\delta}$. Then \eqref{maryam} gives an explicit bound on $r$, that
is%
\begin{equation}
r\leq\frac{c}{\frac{\exp(\delta)}{2^{n-1}}-1}\text{.} \label{cntbd}%
\end{equation}

\end{lemma}

\begin{proof}
Since for $\mathbb{P}$-a.e. $\omega\in\Omega$, $c_{t}(\omega)\leq c$, we get
\[
\int_{\Omega}\log(2^{n-1}(\frac{c_{t}(\omega)}{r}+1))d\mathbb{P(\omega)\leq
}\log(2^{n-1}(\frac{c}{r}+1)).
\]
By \eqref{maryam}, we have $\log(2^{n-1}(\frac{c}{r}+1))\geq\mathbb{\delta}$
which implies%
\begin{equation}
\frac{c}{r}+1\geq\frac{\exp(\mathbb{\delta})}{2^{n-1}}\text{.} \label{chacha}%
\end{equation}
Since $\log(2^{n-1})\mathbb{<\delta}$, we have $\frac{\exp(\mathbb{\delta}%
)}{2^{n-1}}>1$ and thus \eqref{chacha} gives a nontrivial bound on $r$. By
solving \eqref{chacha} for $r$, we get the upper bound given in \eqref{cntbd}.
\end{proof}

The next corollary shows another way of getting finiteness of the number of
measures $r$, previously obtained in Corollary \ref{coro1} using
multiplicative ergodic theory.

\begin{corollary}
Consider the assumptions in Lemma~\ref{lem:bound1}. Then the number of
measures $r$ in Corollary~\ref{coro1} is finite.

\begin{proof}
The integrand in \eqref{maryam} is a non-increasing function of $r$. Hence, as
$r\rightarrow\infty$, we get $\int_{\Omega}\log(2^{n-1})d\mathbb{P(\omega
)\geq\delta}$,which contradicts the assumption.
\end{proof}
\end{corollary}

Another immediate consequence of Lemma~\ref{lem:bound1} is the following.

\begin{corollary}
If $\frac{c}{\frac{\exp(\mathbb{\delta})}{2^{n-1}}-1}<2$, then there exists a
unique ergodic ACIP for the skew product.
\end{corollary}

\subsection{Another bound on $r$}

\label{S:2ndbound} We recall that $\gamma(\omega)$, introduced in
\eqref{shedo}, quantifies the expansion in the random system. The geometry of
the partitions $\{\mathcal{B}^{\omega}\}_{\omega\in\Omega}$ is related to the
quantities $q_{\omega}$, the number of rectangles in the partition
$\mathcal{B}^{\omega}$; $M(\omega)$, defined in \eqref{eq:Mom}; and
$c_{t}(\omega)$, the total number of interior crossing points in the partition
$\mathcal{B}^{\sigma\omega}$.

\begin{lemma}
\label{secondbd}Assume for $\mathbb{P}$-a.e. $\omega\in\Omega$, $M(\omega)\leq
M$, $c_{t}(\omega)\leq c$ and $q_{\omega}\leq q$. Then,
\begin{equation}
r\leq\frac{c(\log(q)-\log(M))}{\int_{\Omega}\log(\gamma(\omega
))d\mathbb{P(\omega)-}\log(M)}\text{.} \label{eman}%
\end{equation}

\end{lemma}

\begin{proof}
Recall that $M<q$, by the definition of $M(\omega)$ in \eqref{eq:Mom}. For
$i=1,2,\dots r$, define%
\[
g_{i}(\omega)=\left\{
\begin{array}
[c]{cc}%
\begin{array}
[c]{c}%
\frac{\mathbb{\gamma}(\omega)}{q_{\omega}}\\
\\
\end{array}
&
\begin{array}
[c]{c}%
:\text{ }\omega\in D_{i}\\
\\
\end{array}
\\
\frac{\mathbb{\gamma}(\omega)}{M(\omega)} & :\text{ }\omega\in\Omega\backslash
D_{i}%
\end{array}
\right.  \text{.}%
\]

In a similar argument to \eqref{ala}, note that for all $\omega\in\Omega$ and
$s=1,2,3,\dots$, we have%
\[
m(I_{i,\omega}(s))\geq g_{i}(\sigma^{s-1}\omega)\text{.}\dots\text{.}%
g_{i}(\omega)m(I_{i,\omega}(0))\text{.}%
\]
Then, for all $s=1,2,3,\dots$, we have%
\begin{equation}
\frac{1}{s}\sum_{k=0}^{s-1}\log(g_{i}(\sigma^{k}\omega))\leq\frac{\log
(\frac{1}{m(I_{i,\omega}(0))})}{s}\text{.} \label{bbb}%
\end{equation}
It is also clear that
\begin{equation}
g_{i}(\omega)\geq\left\{
\begin{array}
[c]{cc}%
\begin{array}
[c]{c}%
\frac{\mathbb{\gamma}(\omega)}{q}\\
\\
\end{array}
&
\begin{array}
[c]{c}%
:\text{ }\omega\in D_{i}\\
\\
\end{array}
\\
\frac{\mathbb{\gamma}(\omega)}{M} & :\text{ }\omega\in\Omega\backslash D_{i}%
\end{array}
\right.  \text{.} \label{hani}%
\end{equation}
Using \eqref{hani} and Birkhoff ergodic theorem, from \eqref{bbb}, we get%
\[
\int_{D_{i}}\log(\frac{\mathbb{\gamma}(\omega)}{q})d\mathbb{P(\omega)+}%
\int_{\Omega\backslash D_{i}}\log(\frac{\mathbb{\gamma}(\omega)}%
{M})d\mathbb{P(\omega)\leq}0\text{,}%
\]
which simplifies to%
\[
\int_{\Omega}\log(\gamma(\omega))d\mathbb{P(\omega)-}\log(M)+m_{i}\log
(\frac{M}{q})\leq0\text{,}%
\]
where $m_{i}=m(D_{i})$. Therefore, for all $i=1,2\dots,r$, we have%
\begin{equation}
m_{i}\geq\frac{\int_{\Omega}\log(\gamma(\omega))d\mathbb{P(\omega)-}\log
(M)}{\log(q)-\log(M)}\text{.} \label{mohd}%
\end{equation}
For $i=1,2\dots,r$, define%
\[
a_{i}(\omega)=\left\{
\begin{array}
[c]{cc}%
1 & :\text{ }\omega\in D_{i}\\
0 & :\text{ }\omega\in\Omega\backslash D_{i}%
\end{array}
\right.  \text{,}%
\]
then for $\mathbb{P}$-a.e. $\omega\in\Omega$, $\sum_{i=1}^{r}a_{i}(\omega)\leq
c$. Note that $m_{i}=\int_{\Omega}a_{i}(\omega)d\mathbb{P(\omega)}$. Taking
the sum over all $i=1,2\dots,r$, we get $r\leq\frac{c}{\min(m_{1},\dots
,m_{r})}$. By \eqref{mohd}, we get the bound given in \eqref{eman}.
\end{proof}

Since $m_{i}\leq1$ for all $i=1,2\dots,r$, an immediate consequence of
\eqref{mohd} is the following.

\begin{corollary}
\label{onavgexpbound1}We have $\int_{\Omega}\log(\gamma(\omega
))d\mathbb{P(\omega)\leq}\log(q)$, where $q$ is defined in Lemma
\ref{secondbd}.
\end{corollary}

\begin{corollary}
If $\frac{c(\log(q)-\log(M))}{\int_{\Omega}\log(\gamma(\omega
))d\mathbb{P(\omega)-}\log(M)}<2$, then there exists a unique ergodic ACIP for
the skew product.
\end{corollary}

\subsection{Example}

\label{S:ex} \begin{figure}[ptb]
\centering
\begin{minipage}{0.45\textwidth}
\centering
\includegraphics[width=0.6\textwidth]{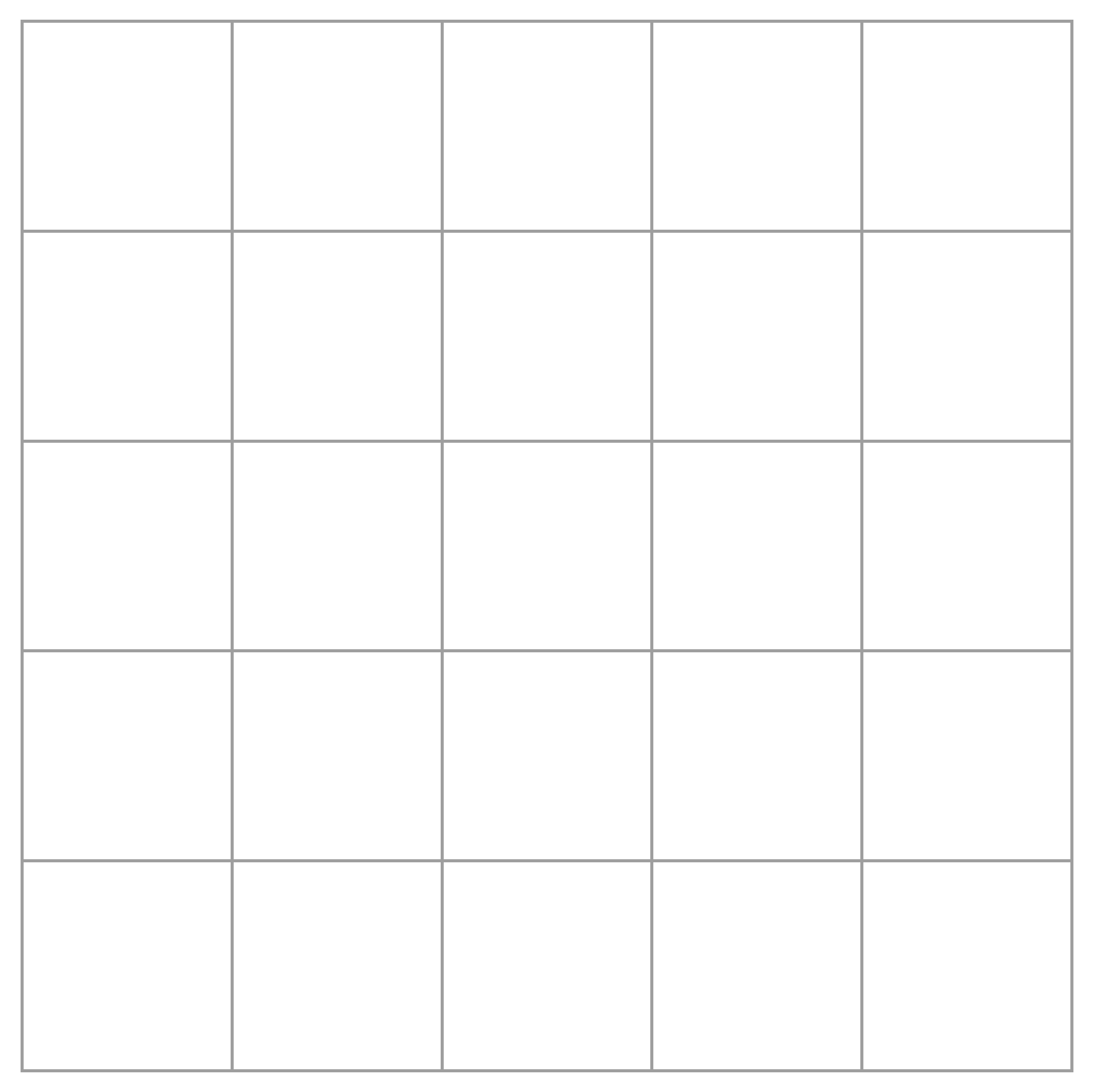} % first figure itself
\caption{$I^{2}$ partitioned into $25$ equal squares.}
\label{partition}
\end{minipage}\hfill\begin{minipage}{0.55\textwidth}
\centering
\includegraphics[width=1.05\textwidth]{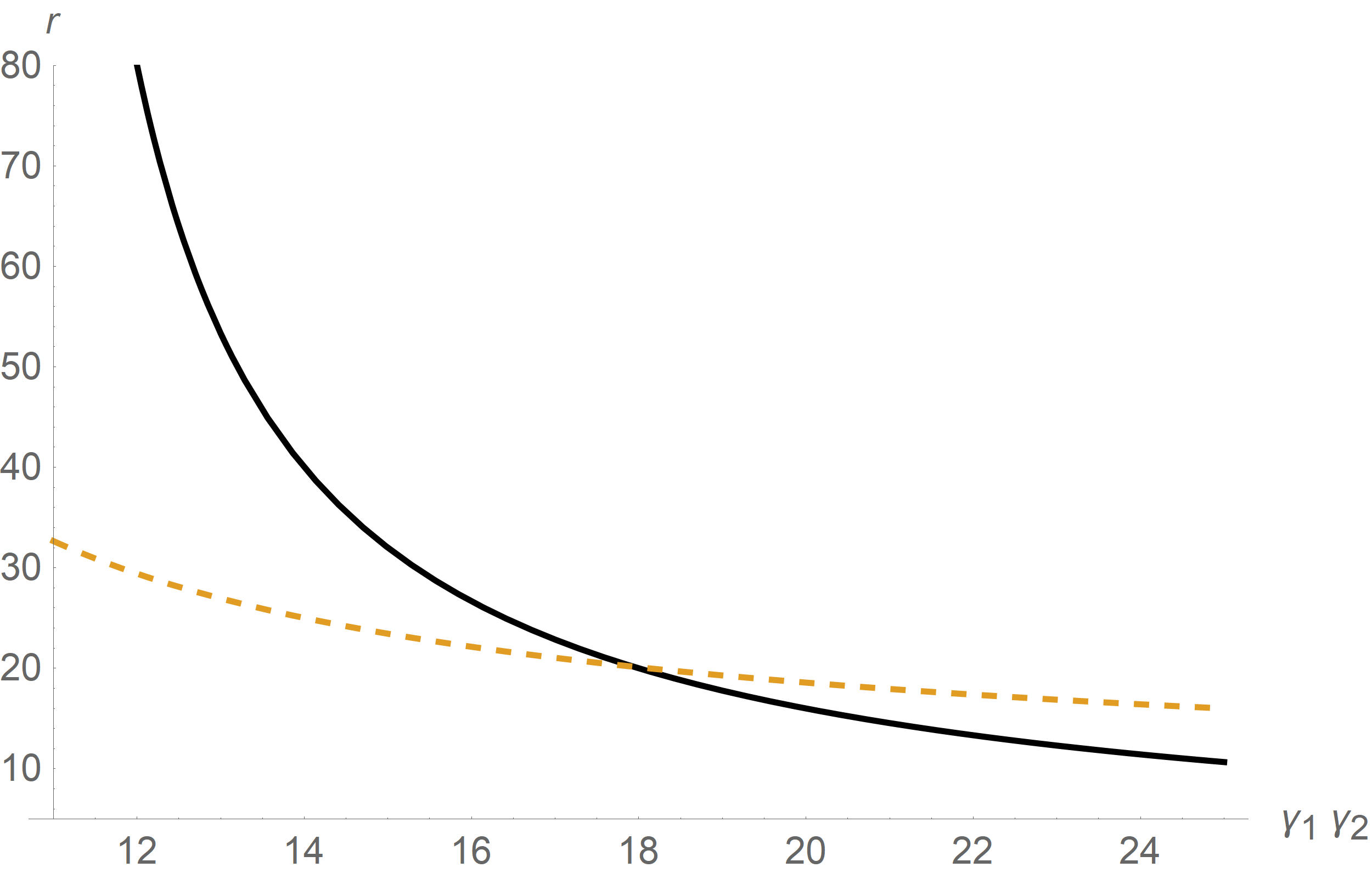} % second figure itself
\caption{Bounds in \eqref{bb1} (solid) and \eqref{ssss} (dashed).}
\label{bounds}
\end{minipage}
\end{figure}

Consider an admissible random Jab\l o\'{n}ski map where the common partition
is taken to be the equally sized $25$ squares partition shown in Figure
\eqref{partition}. For this partition, we have $M=5$, $c=16$ and $q=25$.

Let $\gamma_{1},\gamma_{2}>0$ be such that for all $\omega\in\Omega$,
$\gamma_{1}(\omega)\geq\gamma_{1}$ and $\gamma_{2}(\omega)\geq\gamma_{2}$
where $\gamma_{i}(\omega)$ is defined in \eqref{fec}. By \eqref{shedo}, we
have
\[
\gamma(\omega)=\gamma_{1}(\omega)\gamma_{2}(\omega)\geq\gamma_{1}\gamma
_{2}\text{,}%
\]
for all $\omega\in\Omega$. Note that $\gamma_{1}$ and $\gamma_{2}$ can not
take values such that $\gamma_{1}\gamma_{2}>25$, because the rectangles of the
partition would be mapped outside $I^{2}$. The constant $\delta$ defined in
\eqref{edp} is
\[
\delta=\int_{\Omega}\log(\frac{\mathbb{\gamma}(\omega)}{M(\omega
)})d\mathbb{P(\omega)\geq}\log(\frac{\gamma_{1}\gamma_{2}}{5})\text{.}%
\]
For the bound in Section~\ref{S:1stbound}, we must have, in addition, that
$\gamma_{1}$ and $\gamma_{2}$ can not be such that $\gamma_{1}\gamma_{2}%
\leq10$. Since this contradicts the condition in Lemma \ref{bbbb} that
$\log(2^{n-1})\mathbb{<\delta}$, we can make the restriction that
\[
10<\gamma_{1}\gamma_{2}\leq25.
\]
Then, \eqref{cntbd} implies that
\begin{equation}
r\leq\frac{160}{\gamma_{1}\gamma_{2}-10}. \label{bb1}%
\end{equation}
The bound in \eqref{eman}, implies that
\begin{equation}
r\leq\frac{16\log(5)}{\log(\frac{\gamma_{1}\gamma_{2}}{5})}\text{.}
\label{ssss}%
\end{equation}

Figure \eqref{bounds} shows the dependence of the two bounds on $\gamma
_{1}\gamma_{2}$ and the regions on which each of the bounds is sharper. The
bounds from Sections~\ref{S:1stbound} and \ref{S:2ndbound} are shown in
black/solid and orange/dashed, respectively.

\section{Acknowledgments}

The authors are thankful to the referees for helpful suggestions and
corrections. The authors are thankful to Prof. Luigi Ambrosio for valuable
comments regarding the space of bounded variation defined in different
equivalent forms, Prof. Chris Bose for useful conversations and Prof. Sandro
Vaienti for bringing reference \cite{NakamuraToyokawa} to our attention. The
authors have been partially supported by the Australian Research Council
(DE160100147). Fawwaz Batayneh acknowledges the support of the University of
Queensland through an Australian Government Research Training Program Scholarship.

\end{document}